\documentclass[10pt,leqno,twoside,draft]{amsart}
\usepackage{amsmath}
\usepackage{amsfonts}
\usepackage{enumerate}
\usepackage{amsthm}
\usepackage{bm}
\usepackage{mathtools}
\usepackage{color}
\makeatletter \@addtoreset{equation}{section} \makeatother
\newcommand{\eref}[1]{(\ref{#1})}
\newcommand{\tref}[1]{Theorem \ref{#1}}

\newcommand{\lref}[1]{Lemma \ref{#1}}
\newcommand{\rref}[1]{Remark \ref{#1}}
\newcommand{\sref}[1]{Section \ref{#1}}

\newcommand{\jump}[1]{[#1]}
\newcommand{\vertiii}[1]{{\left\vert\kern-0.25ex\left\vert\kern-0.25ex\left\vert #1 \right\vert\kern-0.25ex\right\vert\kern-0.25ex\right\vert}}
\theoremstyle{plain} \newtheorem{thm}{Theorem}[section] \newtheorem{lem}{Lemma}[section]  \newtheorem{cor}{Corollary}[section]
\theoremstyle{definition} \newtheorem{rem}{Remark}[section]
\title[Finite element approximation for the Stokes equations]{Penalty method with Crouzeix--Raviart approximation for\\the Stokes equations under slip boundary condition}
\author[Takahito Kashiwabara]{Takahito Kashiwabara}
\address{The Graduate School of Mathematical Sciences, The University of Tokyo, 3-8-1 Komaba, Meguro, Tokyo 153-8914, Japan}
\email{tkashiwa@ms.u-tokyo.ac.jp}
\author[Issei Oikawa]{Issei Oikawa}
\address{Faculty of Science and Engineering, Waseda University, 3-4-1 Okubo, Shinjuku, Tokyo 169-8555, Japan}
\email{oikawa@aoni.waseda.jp}
\author[Guanyu Zhou]{Guanyu Zhou}
\address{Department of Applied Mathematics, Tokyo University of Science, 1-3 Kagurazaka, Shinjuku, Tokyo 162-8601}
\email{zhoug@rs.tus.ac.jp }
\date{\today}
\thanks{This study was supported by JSPS Grant-in-Aid for Young Scientists B (17K14230, 17K14243) and by JSPS Grant-in-Aid for Early-Career Scientists (18K13460).}

\begin{document}
\begin{abstract}
	The Stokes equations subject to non-homogeneous slip boundary conditions are considered in a smooth domain $\Omega \subset \mathbb R^N \, (N=2,3)$.
	We propose a finite element scheme based on the nonconforming P1/P0 approximation (Crouzeix--Raviart approximation) combined with a penalty formulation and with reduced-order numerical integration in order to address the essential boundary condition $u \cdot n_{\partial\Omega} = g$ on $\partial\Omega$.
	Because the original domain $\Omega$ must be approximated by a polygonal (or polyhedral) domain $\Omega_h$ before applying the finite element method, we need to take into account the errors owing to the discrepancy $\Omega \neq \Omega_h$, that is, the issues of domain perturbation.
	In particular, the approximation of $n_{\partial\Omega}$ by $n_{\partial\Omega_h}$ makes it non-trivial whether we have a discrete counterpart of a lifting theorem, i.e., right-continuous inverse of the normal trace operator $H^1(\Omega)^N \to H^{1/2}(\partial\Omega)$; $u \mapsto u\cdot n_{\partial\Omega}$.
	In this paper we indeed prove such a discrete lifting theorem, taking advantage of the nonconforming approximation, and consequently we establish the error estimates $O(h^\alpha + \epsilon)$ and $O(h^{2\alpha} + \epsilon)$ for the velocity in the $H^1$- and $L^2$-norms respectively, where $\alpha = 1$ if $N=2$ and $\alpha = 1/2$ if $N=3$.
	This improves the previous result [T. Kashiwabara et al., Numer.\ Math.\ 134 (2016), pp.\ 705--740] obtained for the conforming approximation in the sense that there appears no reciprocal of the penalty parameter $\epsilon$ in the estimates.
\end{abstract}
\maketitle

\section{Introduction}
This work is continuation of \cite{KOZ16} and we consider the same PDEs as there, that is, the slip boundary value problem of the Stokes equations in a bounded smooth domain $\Omega\subset\mathbb R^N$ as follows:
\begin{equation} \label{eq: Stokes slip BC}
	\left\{
	\begin{aligned}
		u - \nu \Delta u + \nabla p &= f \quad\text{in}\quad \Omega, \\
		\mathrm{div}\, u &= 0 \quad\text{in}\quad \Omega, \\
		u\cdot n &= g \quad\text{on}\quad \Gamma := \partial\Omega, \\
		(\mathbb I - n\otimes n) \sigma(u, p) n &= \tau \quad\text{on}\quad \Gamma.
	\end{aligned}
	\right.
\end{equation}
As in \cite{KOZ16}, $\nu>0$ is a viscosity constant, $n$ means the outer unit normal to $\Gamma$, and $\sigma(u, p) := -p \mathbb I + \nu(\nabla u + (\nabla u)^\top)$ denotes the stress tensor.
We impose the compatibility condition between \eref{eq: Stokes slip BC}$_2$ and \eref{eq: Stokes slip BC}$_3$ by
\begin{equation} \label{eq: compatibility of g}
	\int_\Gamma g \, ds = 0.
\end{equation}
The first term of \eref{eq: Stokes slip BC}$_1$ is added in order to avoid cumbersomeness concerning rigid body rotations (see \cite[Remark 1.1]{KOZ16}).

Before explaining the goals of the present paper, let us review the results of \cite{KOZ16}.
Since the original domain $\Omega$ has a curved boundary, we need to approximate it by a polygonal or polyhedral domain $\Omega_h$ to invoke the finite element method, where we construct meshes, build finite element spaces, and define variational formulations.
In case of the slip boundary problem, however, one has to be careful in setting a test function space.
In fact, imposing the constraint $v_h \cdot n_h = 0$ at each degree of freedom on $\Gamma_h$ ($n_h$ being the outer unit normal to $\Gamma_h$), which seems natural at first glance, would result in a variational crime.
Several strategies to overcome it are proposed e.g.\ in \cite{BaDe99, DiUr15, Kno99, Ver87}.

In \cite{KOZ16}, we considered to weakly impose the constraint above by the penalty method together with reduced-order numerical integration.
Employing the P1/P1 approximation, we derived the error bound $O(h + \epsilon^{1/2} + h^{2\alpha}/\epsilon^{1/2})$ for the $H^1$- and $L^2$-norms of velocity and pressure, respectively.
Here, $h$ and $\epsilon$ denote the discretization and penalty parameters, respectively, and the number $\alpha$ is given by $\alpha = 1$ if $N=2$ and $\alpha = 1/2$ if $N=3$.
In particular, the optimal rate of convergence $O(h)$ was achieved by choosing $\epsilon = O(h^2)$ in the two-dimensional case.
This strategy was then extended to the stationary Navier--Stokes equations in \cite{ZKO16} and to the non-stationary Stokes equations in \cite{ZKO17}.

The first goal of the present paper is to improve the error bound mentioned above.
In fact, the rate $O(h + \epsilon^{1/2} + h^{2\alpha}/\epsilon^{1/2})$ is not optimal because it is known that the penalty method admits the optimal rate of convergence $O(h + \epsilon)$ for polygonal or polyhedral domains, i.e., when $\Omega = \Omega_h$ (see \cite{CaLi09}).
We show that the nonconforming P1/P0 approximation (also known as the Crouzeix--Raviart approximation, see \cite{BuHa05, CrRa73}) for smooth domains, combined with the penalty method and with reduced-order numerical integration, leads to the rate $O(h^{\alpha} + \epsilon)$, where the meaning of $\alpha$ is the same as above.
Therefore, for the two-dimensional case we establish the optimal rate $O(h + \epsilon)$ even when $\Omega \neq \Omega_h$.
Moreover, we also provide the $L^2$-error estimate for velocity, giving the rate of convergence $O(h^{2\alpha} + \epsilon)$, which was not available in \cite{KOZ16}.

The key point of our approach is that, in the Crouzeix--Raviart approximation, the degrees of freedom for velocity (namely, the midpoints of edges or the barycenters of faces) agree with those of $n_h$ on the boundary $\Gamma_h$.
This fact enables us to prove a discrete counterpart to the inf-sup condition
\begin{equation*}
	C \|\mu\|_{H^{-1/2}(\Gamma)} \le \sup_{v \in H^1(\Omega)} \frac{\int_\Gamma (v\cdot n) \mu \, ds}{\|v\|_{H^1(\Omega)}} \qquad \forall \mu \in H^{-1/2}(\Gamma),
\end{equation*}
which was not available for the P1/P1 approximation in \cite{KOZ16}.
This follows from a discrete counterpart of a lifting theorem, more precisely, a stability estimate concerning a right continuous inverse of the trace operator in the normal direction:
\begin{equation*}
	H^1(\Omega)^N \to H^{1/2}(\Gamma); \quad v \mapsto v|_{\Gamma} \cdot n.
\end{equation*}
We emphasize, however, that such a discrete lifting theorem in $\Omega_h$ is completely non-trivial since $n_h$, which is only piecewisely constant on $\Gamma_h$, has jump discontinuities and thus fails to belong to $H^{1/2}(\Gamma_h)^N$.
Similarly, the trace of a nonconforming P1 function $v_h$ to the boundary does not necessarily admit $H^{1/2}$-regularity (cf.\ \cite[Appendix]{BHP05}).
To overcome those difficulties, we introduce a discrete version of the $H^{1/2}(\Gamma_h)$-norm and combine it with the so called \emph{enriching operator} (cf.\ \cite[Appendix B]{BSZ14}) to reduce the nonconforming approximation to the conforming one, which is a basic strategy to prove the discrete lifting theorem.

The second goal of the present paper is to provide, in case of nonconforming approximations, a framework to address the errors owing to the discrepancy $\Omega \neq \Omega_h$, which we refer to as \emph{domain perturbation}.
To the best of our knowledge, there are very few studies in the literature dealing with the issues of domain perturbation when nonconforming approximations, including discontinuous Galerkin methods, are involved.
However, nonconforming approximations in the situation of domain perturbation is important when considering interfacial transmission problems (an example is the Stokes--Darcy problem, see e.g.\ \cite{BHP05, LSY03}).
In fact, for such problems it is natural to encounter physical jump discontinuities in normal or tangential directions along curved interfaces, which could be treated by the use of nonconforming approximations.
In future work, we would like to extend the techniques developed in this paper to interface problems in dealing with domain perturbation.

The rest of this paper is organized as follows.
In \sref{sec2} we introduce variational formulation, triangulation, and finite element spaces.
We also propose our finite element scheme and state the main results.
In \sref{sec3}, auxiliary lemmas relating to the discrete $H^{1/2}$-norm and to domain perturbation estimates are stated.
Some of their proofs will be given in Appendices.
After establishing discrete well-posedness in \sref{sec4}, we derive the $H^1$- and $L^2$-error estimates (for velocity) in Sections \ref{sec5} and \ref{sec6}, respectively.
We give a numerical example in \sref{sec7} to confirm the theoretical result.
Throughout this paper, $C$ will denote a generic constant which may depend only on $\Omega$, $N$, and $\nu$ unless otherwise stated.

\section{Preliminaries and Main Theorem} \label{sec2}
\subsection{Function spaces and variational forms}
Throughout this paper, we adopt the standard notion of Lebesgue and Sobolev spaces.
To state a variational formulation for \eref{eq: Stokes slip BC}, we set
\begin{align*}
	V = H^1(\Omega)^N, \quad Q = L^2(\Omega), \quad \mathring V = H^1_0(\Omega)^N, \quad \mathring Q = L^2_0(\Omega),
\end{align*}
and
\begin{equation*}
	V_n = \{ v \in V \,:\, v\cdot n = 0 \; \text{on} \; \Gamma \}.
\end{equation*}
Next, for a domain $G\subset\mathbb R^N$ we define bilinear forms as follows:
\begin{align*}
	a_G(u, v) &= (u, v)_G + \frac\nu2 (\mathbb E(u), \mathbb E(v))_G, \\
	b_G(p, v) &= -(p, \mathrm{div}\, v)_G, \\
	c_{\partial G} (\lambda, \mu) &= (\lambda, \mu)_{\partial G},
\end{align*}
where $\mathbb E(u) := \nabla u + (\nabla u)^\top$ and $(\cdot, \cdot)_G$ denotes the inner product of $L^2(G)$.

The weak form for \eref{eq: Stokes slip BC} now reads as follows: find $(u, p) \in V\times \mathring Q$ satisfying $u\cdot n = g$ on $\Gamma$ and
\begin{equation} \label{eq: 2-field formulation}
	\left\{
	\begin{aligned}
		a(u, v) + b(p, v) &= (f, v)_\Omega + (\tau, v)_\Gamma \quad &&\forall v \in V_n, \\
		b(q, u) &= 0 \quad &&\forall q \in\mathring Q, \\
	\end{aligned}
	\right.
\end{equation}
where we have employed the abbreviations $a := a_{\Omega}$ and $b := b_{\Omega}$.
Defining the Lagrange multiplier $\lambda := -\sigma(u,p)n\cdot n \in H^{-1/2}(\Gamma) =: \Lambda$, one sees that $(u, p, \lambda)$ satisfies
\begin{equation} \label{eq: 3-field formulation}
	\left\{
	\begin{aligned}
		a(u, v) + b(p, v) + c(\lambda, v\cdot n) &= (f, v)_\Omega + (\tau, v)_\Gamma \quad &&\forall v \in V, \\
		b(q, u) &= 0 \quad &&\forall q \in Q, \\
		c(\mu, u\cdot n - g) &= 0 \quad &&\forall \mu \in \Lambda,
	\end{aligned}
	\right.
\end{equation}
where $c$ means $c_{\partial\Omega}$.
The well-posedness of \eref{eq: Stokes slip BC} (or \eref{eq: 2-field formulation}, \eref{eq: 3-field formulation}) is well known e.g.\ in \cite{BdV04}; in particular, if $f \in L^2(\Omega)$, $g \in H^{3/2}(\Gamma)$, and $\tau \in H^{1/2}(\Gamma)^N$, then there exists a unique solution such that $u \in H^2(\Omega)^N$ and $p \in H^1(\Omega) \cap L^2_0(\Omega)$.

\subsection{Triangulations}
Let $\{\mathcal T_h\}_{h\downarrow0}$ be a regular family of triangulations of a polyhedral domain $\Omega_h$, which is assigned the \emph{mesh size} $h>0$.
Namely, we assume that:
\begin{enumerate}[(H1)]
	\item each $T\in\mathcal T_h$ is a closed $N$-simplex such that $h_T := \mathrm{diam}\,T \le h$;
	\item $\Omega_h = \bigcup_{T\in\mathcal T_h}T$;
	\item the intersection of any two distinct elements is empty or consists of their common face of dimension $\le N-1$;
	\item there exists a constant $C>0$, independent of $h$, such that $\rho_T \ge Ch_T$ for all $T\in\mathcal T_h$ where $\rho_T$ denotes the diameter of the inscribed ball of $T$.
\end{enumerate}
Moreover, we denote by $\mathcal E_h$ the set of the edges or faces, that is,
\begin{equation*}
	\mathcal E_h = \{e\subset\overline\Omega_h \,:\, \text{$e$ is an $(N-1)$-dimensional face of some $T \in \mathcal T_h$} \}.
\end{equation*}
The sets of the interior and boundary edges are denoted by $\mathring{\mathcal E}_h$ and $\mathcal E_h^\partial$ respectively, namely,
\begin{equation*}
	\mathcal E_h^\partial = \{ e\in\mathcal E_h \,:\, e\subset\Gamma_h \}, \qquad \mathring{\mathcal E}_h = \mathcal E_h\setminus\mathcal E_h^\partial.
\end{equation*}
We assume that $\Omega_h$ approximates $\Omega$ in the following sense:
\begin{enumerate}[(H1)]
	\setcounter{enumi}{4}
	\item the vertices of every $e \in \mathcal E_h^\partial$ lie on $\Gamma = \partial\Omega$.
\end{enumerate}
Throughout this paper, we confine ourselves to the case where $0 < h \ll 1$ is sufficiently small, which will not be emphasized below.

The set of vertices and that of midpoints of edges are defined as
\begin{equation*}
	\mathcal V_h = \{ p\in\overline\Omega_h \,:\, \text{$p$ is a vertex of some $T \in \mathcal T_h$} \}, \qquad \mathcal M_h = \{ m_e \in \overline\Omega_h \,:\, e \in \mathcal E_h \}, 
\end{equation*}
where $m_e$ means the midpoint (barycenter) of $e\in\mathcal E_h$.
We introduce a broken Sobolev space by
\begin{equation*}
	H^1(\mathcal T_h) = \{v \in L^2(\Omega_h) \,:\, v|_T \in H^1(T) \; \forall T \in \mathcal T_h\}.
\end{equation*}
To describe jump discontinuities across interior edges, for $v \in H^1(\mathcal T_h)$ we define
\begin{equation*}
	\jump{v}(x) := \lim_{s\to0+} (v(x + sn_e) - v(x - sn_e)), \quad x \in e \in \mathring{\mathcal E}_h,
\end{equation*}
where $n_e$ is a unit normal vector to $e$.
For $e \in \mathring{\mathcal E}_h$ (resp.\ $e \in \mathcal E_h^\partial$), there exists a unique element $T_e^{\pm} \in \mathcal T_h$ (resp.\ $T \in \mathcal T_h$) such that $m_e \pm s n_e \in T_e^{\pm}$ with sufficiently small $s>0$ (resp.\ $m_e \in T_e$).
\begin{rem}
	There are two choices for the direction of $n_e$.
	In this paper, we suppose that each $e \in \mathring{\mathcal E}_h$ is given an arbitrary orientation, which determines the direction of $n_e$.
	Note that, given a vector function $v$, the jump term $[v\cdot n_e](x)$ is well defined regardless of the orientation.
\end{rem}

\subsection{Crouzeix--Raviart element}
For each $T\in\mathcal T_h$ we denote by $P_k(T)$ the space of the polynomial functions of degree up to $k$ defined in $T$.
In the Crouzeix--Raviart element, velocity and pressure are approximated by nonconforming P1 and P0 functions, respectively.  Thereby we introduce
\begin{align*}
	V_h &= \{ v_h \in H^1(\mathcal T_h)^N \,:\, v_h|_T \in P_1(T)^N \; \forall T \in \mathcal T_h, \quad \jump{v_h} (m_e) = \textstyle\frac1{|e|}\int_e \jump{v_h} \,ds = 0 \; \forall e \in \mathring{\mathcal E}_h \}, \\
	Q_h &= \{ q_h \in L^2(\Omega_h) \,:\, v_h|_T \in P_0(T) \; \forall T \in \mathcal T_h \},
\end{align*}
where $|e|$ stands for the $(N-1)$-dimensional measure of $e$.
We will also utilize the conforming P1 finite element space, that is,
\begin{equation*}
	\overline V_h = \{ v_h \in C(\overline\Omega_h)^N \,:\, v_h|_T \in P_1(T)^N \; \forall T \in \mathcal T_h \}.
\end{equation*}

The nodal basis functions of $V_h$ and $\overline V_h$ are denoted by $\{\phi_e\}_{e \in \mathcal E_h}$ and $\{\bar\phi_p\}_{p \in \mathcal V_h}$ respectively, where $\phi_e \in V_h$ and $\bar\phi_p \in \overline V_h$ are defined by the conditions
\begin{equation*}
	\phi_e(x) = \begin{cases} 1 & \text{if } x=m_e, \\ 0 & \text{if } x\neq m_e, e\in \mathcal E_h, \end{cases} \qquad
	\bar\phi_p(x) = \begin{cases} 1 & \text{if } x=p, \\ 0 & \text{if } x\neq p, x\in \mathcal V_h. \end{cases}
\end{equation*}
It follows from \cite[Theorem 3.1.2]{Cia78} and regularity of meshes that
\begin{align*}
	\|\phi_e\|_{H^m(T)} &\le Ch_e^{N/2 - m}, && e \in \mathcal E_h, \, T \in \mathcal T_h, \, e \cap T \neq \emptyset, \\
	\|\phi_e\|_{H^m(e')} &\le Ch_e^{(N-1)/2 - m}, && e, e' \in \mathcal E_h, \, e \cap e' \neq \emptyset,
\end{align*}
where $h_e := \operatorname{diam}e$, and the quantities dependent only on a fixed reference element (e.g.\ unit simplex) are combined into generic constants $C$.
Similar estimates also hold for nodal basis functions $\bar\phi_p$ of $\overline V_h$, provided that the vertex $p$ belongs to $T \in \mathcal T_h$ or $e' \in \mathcal E_h$.

Approximate spaces for $\mathring V$ and $\mathring Q$ are given as
\begin{equation*}
	\mathring V_h = \{ v_h \in V_h \,:\, v_h(m_e) = 0 \; \forall e \in \mathcal E_h^\partial \}, \qquad \mathring Q_h = Q_h \cap \mathring Q.
\end{equation*}
We note, however, that $v_h \in \mathring V_h$ does not imply $v_h|_{\Gamma_h} \equiv 0$.
We equip $V_h$ and $Q_h$ with the norms
\begin{equation*}
	\|v_h\|_{V_h} = \Big( \|v_h\|_{L^2(\Omega_h)}^2 + \sum_{T \in \mathcal T_h} \|\nabla v_h\|_{L^2(T)}^2 \Big)^{1/2}, \qquad \|q_h\|_{Q_h} = \|q_h\|_{L^2(\Omega_h)}.
\end{equation*}
To describe Lagrange multipliers defined on $\Gamma_h$, we set
\begin{align*}
	\Lambda_h &= \{\mu_h \in L^2(\Gamma_h) \,:\, \mu_h|_e \in P_0(e) \; \forall e \in \mathcal E_h^\partial\}, \\
	\overline\Lambda_h &= \{\mu_h \in C(\Gamma_h) \,:\, \mu_h|_e \in P_1(e) \; \forall e \in \mathcal E_h^\partial\}.
\end{align*}

An interpolation operator $\Pi_h : H^1(\Omega_h)^N \to V_h$ is defined by $\Pi_h v(m_e) = \frac1{|e|} \int_e v \, ds$ for $e \in \mathcal E_h$.
It is known (see \cite{CrRa73}) that
\begin{align*}
	\|v - \Pi_h v\|_{L^2(T)} + h_T \|\nabla(v - \Pi_h v)\|_{L^2(T)} &\le Ch_T^2 \|\nabla^2 v\|_{L^2(T)}, \qquad T \in \mathcal T_h, \; v \in H^2(T), \\
	\|v - \Pi_h^\partial v\|_{H^{-1/2}(e)} &\le Ch_e \|\nabla v\|_{L^2(T_e)}, \qquad e \in \mathcal E_h^\partial, \; v \in H^1(T_e).
\end{align*}
For convenience, we also define an analogue of $\Pi_h$ restricted to the boundary, namely, we define $\Pi_h^\partial: L^2(\Gamma_h) \to \Lambda_h$ by $\Pi_h^\partial v(m_e) = \frac1{|e|} \int_e v\, ds$ for all $e \in \mathcal E_h^\partial$.

The continuity at the midpoints ensures that
\begin{equation} \label{eq: equivalence of double-line and triple-line norms}
	\sum_{e \in \mathring{\mathcal E}_h} h_e^{-1} \|\jump{v_h}\|_{L^2(e)}^2 \le C \sum_{T \in \mathcal T_h} \|\nabla v_h\|_{L^2(T)}^2 \quad \forall v_h \in V_h.
\end{equation}
In fact, since $\jump{v_h}(m_e) = 0$ for $e \in \mathring{\mathcal E}_h$ and $\nabla v_h$ is piecewisely constant, we have
\begin{align*}
	\frac1{h_e} \int_e |\jump{v_h}|^2\, ds &\le \frac1{2h_e} \left( \int_e \big| v_h|_{T_e^+} - v_h(m_e) \big|^2\, ds + \int_e \big| v_h|_{T_e^-} - v_h(m_e) \big|^2\, ds \right) \\
		&\le \frac{|e| h_e}2 (\|\nabla v_h\|_{L^\infty(T_e^{+})}^2 + \|\nabla v_h\|_{L^\infty(T_e^{-})}^2) \le C \sum_{T \in \mathcal T_h(e)} \|\nabla v_h\|_{L^2(T)}^2,
\end{align*}
which after the summation for $e\in \mathring{\mathcal E}_h$ proves \eref{eq: equivalence of double-line and triple-line norms}.
Hence $\|\cdot\|_{V_h}$ is equivalent to $\vertiii{\,\cdot\,}_{V_h}$ given by
\begin{equation*}
	\vertiii{v_h}_{V_h} = \Big( \|v_h\|_{L^2(\Omega_h)}^2 + \sum_{T \in \mathcal T_h} \|\nabla v_h\|_{L^2(T)}^2 + \sum_{e \in \mathring{\mathcal E}_h} h_e^{-1} \|\jump{v_h}\|_{L^2(e)}^2 \Big)^{1/2}, \quad v_h \in V_h,
\end{equation*}
which often appears in discontinuous Galerkin methods.

Adding up the trace inequality $\|v\|_{L^2(e)} \le C \|v\|_{L^2(T_e)}^{1/2} \|v\|_{H^1(T_e)}^{1/2}$ for $e \in \mathcal E_h^\partial$ yields
\begin{equation*}
	\|v\|_{L^2(\Gamma_h)} \le C \|v\|_{L^2(\Omega_h)}^{1/2} \|v\|_{V_h}^{1/2}, \quad v \in H^1(\mathcal T_h),
\end{equation*}
where the constant $C$ depends only on a reference element.

An interpolation operator for pressure is defined as the projector $R_h : Q \to Q_h$, that is, $(R_h p - p, q_h)_{\Omega_h} = 0$ for all $p \in Q$ and $q_h \in Q_h$.
Then we have (see \cite[Lemma 12.4.3]{BrSc07})
\begin{equation*}
	\|R_h p - p\|_{Q_h} \le Ch \|\nabla p\|_{L^2(\Omega_h)}, \qquad p \in H^1(\Omega_h).
\end{equation*}
We also note that $R_h(\mathring Q) \subset \mathring Q_h$.

\subsection{FE scheme with penalty and main theorem}
We propose a finite element approximate problem to \eref{eq: Stokes slip BC} as follows: choose $\epsilon>0$ and find $(u_h, p_h) \in V_h\times Q_h$ such that
\begin{equation} \label{eq: discrete 2-field formulation}
	\left\{
	\begin{aligned}
		&a_h(u_h, v_h) + b_h(p_h, v_h) + \frac1\epsilon c_h(u_h\cdot n_h - \tilde g, v_h\cdot n_h) + j_h(u_h, v_h) = (\tilde f, v_h)_{\Omega_h} + (\tilde\tau, v_h)_{\Gamma_h} && \forall v_h \in V_h, \\
		&b_h(q_h, u_h) = 0 && \forall q_h \in Q_h.
	\end{aligned}
	\right.
\end{equation}
Here, we are making use of an extension operator $P: W^{m,p}(\Omega) \to W^{m,p}(\mathbb R^N)$ satisfying the stability condition $\|Pv\|_{W^{m,p}(\mathbb R^N)} \le C\|v\|_{W^{m,p}(\Omega)}$, where the constant $C$ depends only on $N$, $\Omega$, $m$, and $p$.
If this is combined with a stable lifting operator (right continuous inverse of the trace operator) $L: W^{m-1/p, p}(\Gamma) \to W^{m,p}(\Omega) \, (m\ge 1)$, one can also consider extensions from $\Gamma$ to $\mathbb R^N$.
In the following, all of such extensions are simply denoted by $\tilde f$, $\tilde g$, $\tilde \tau$, etc.
\begin{rem}
	The way of extensions may be arbitrary as far as they satisfy the stability conditions mentioned above.
	In particular, $P$ or $L$ has no effect on the rate of convergence in Theorems \ref{thm: H1 error estimate} and \ref{thm: L2 error estimate}, whereas the constants $C$ appearing there will depend on the choice of them.
\end{rem}

The bilinear forms in \eref{eq: discrete 2-field formulation} are defined by
\begin{align*}
	a_h(u, v) &= \sum_{T \in \mathcal T_h} \Big( (u, v)_T + \frac\nu2 (\mathbb E(u), \mathbb E(v))_T \Big), && u, v \in H^1(\mathcal T_h), \\
	b_h(p, v) &= - \sum_{T \in \mathcal T_h} (p, \operatorname{div}v)_T, && p \in Q, v \in H^1(\mathcal T_h), \\
	c_h(\lambda, \mu) & = (\Pi_h^\partial \lambda, \Pi_h^\partial \mu)_{\Gamma_h}, && \lambda, \mu \in L^2(\Gamma_h), \\
	j_h(u, v) &= \sum_{e \in \mathring{\mathcal E}_h} \frac{\gamma}{h_e} ([u], [v])_e, && u, v \in H^1(\mathcal T_h),
\end{align*}
where $\gamma$ is a stabilization parameter, which one can choose to be any positive constant.
\begin{rem}
	For $u_h, v_h \in V_h$, we see that $c_h(u_h\cdot n_h, v_h\cdot n_h)$ agrees with the midpoint (barycenter) formula applied to $(u_h\cdot n_h, v_h\cdot n_h)_{\Gamma_h}$.
	In this sense, reduced-order numerical integration is applied to the penalty term.
\end{rem}

The main results of this paper are the well-posedness and error estimates to \eref{eq: discrete 2-field formulation} stated as follows.
\begin{thm} \label{thm: well-posedness}
	There exists a unique solution $(u_h, p_h) \in V_h \times Q_h$ of \eref{eq: discrete 2-field formulation}.
	Moreover, it satisfies
	\begin{align}
		\|u_h\|_{V_h} + \|\mathring p_h\|_{Q_h} &\le C \big( \|f\|_{L^2(\Omega)} + \|\tau\|_{H^{1/2}(\Gamma)} + (1 + h\epsilon^{-1/2}) \|g\|_{H^{3/2}(\Gamma)} \big), \label{eq: stability of uh and ph} \\
		|k_h| &\le C \big( \|f\|_{L^2(\Omega)} + \|\tau\|_{H^{1/2}(\Gamma)} + (1 + h^2\epsilon^{-1}) \|g\|_{H^{3/2}(\Gamma)} \big), \label{eq: stability of kh}
	\end{align}
	where $k_h := (p_h, 1)_{\Omega_h}/|\Omega_h|$ and $\mathring p_h := p_h - k_h \in \mathring Q_h$.
\end{thm}
\begin{rem} \label{rem: stability}
	(i) If $g = 0$, the terms involving $\epsilon^{-1}$ do not appear.
	
	(ii) Even if $g \neq 0$, $\|u_h\|_{V_h}$ becomes independent of $\epsilon \le 1$ in the end as a consequence of \tref{thm: H1 error estimate}.
\end{rem}

\begin{thm} \label{thm: H1 error estimate}
	Let $(u, p) \in H^2(\Omega)^N\times H^1(\Omega)$ be the solution of \eref{eq: Stokes slip BC} and $(u_h, p_h) \in V_h\times Q_h$ be that of \eref{eq: discrete 2-field formulation}.
	Then we obtain
	\begin{equation*}
		\vertiii{\tilde u - u_h}_{V_h} + \|\tilde p - \mathring p_h\|_{Q_h} \le C(h^\alpha + \epsilon) (\|f\|_{L^2(\Omega)} + \|g\|_{H^{3/2}(\Gamma)} + \|\tau\|_{H^{1/2}(\Gamma)}),
	\end{equation*}
	where $\alpha = 1$ if $N=2$ and $\alpha = 1/2$ if $N=3$.
\end{thm}
\begin{thm} \label{thm: L2 error estimate}
	Under the same assumption as in the previous theorem, we obtain
	\begin{equation*}
		\|\tilde u - u_h\|_{L^2(\Omega_h)} \le C(h^{2\alpha} + \epsilon) (\|f\|_{L^2(\Omega)} + \|g\|_{H^{3/2}(\Gamma)} + \|\tau\|_{H^{1/2}(\Gamma)}).
	\end{equation*}
\end{thm}
The proofs of Theorems \ref{thm: well-posedness}--\ref{thm: L2 error estimate} will be given in Sections \ref{sec4}--\ref{sec6}, respectively.

\section{Auxiliary Lemmas} \label{sec3}
\subsection{Discrete $H^{1/2}$-norm}
It is well known that there exists a right continuous inverse of the trace operator $H^1(\Omega)^N \to H^{1/2}(\Gamma) ; \; v \mapsto (v\cdot n)|_\Gamma$, which we call a \emph{lifting operator} with respect to the normal component.
We need its analogue in the Crouzeix--Raviart element case.
However, since functions having jump discontinuities do not belong to $H^{1/2}$, we devise a discrete $H^{1/2}(\Gamma_h)$-norm for $\mu_h \in \Lambda_h$ as follows:
\begin{equation*}
	\|\mu_h\|_{1/2, \Lambda_h} = \Big( \| E_h^\partial \mu_h \|_{H^{1/2}(\Gamma_h)}^2 + \sum_{e \in \mathcal E_h^\partial} \sum_{e' \in \mathcal E_h^\partial(e)} h_e^{N-2} |\mu_h(m_e) - \mu_h(m_{e'})|^2 + h\|\mu_h\|_{ L^2(\Gamma_h) }^2 \Big)^{1/2}.
\end{equation*}
Here, $E_h^\partial: \Lambda_h \to \overline\Lambda_h$ is a kind of \emph{enriching operators} (cf.\ \cite[Appendix B]{BSZ14}) defined by
\begin{equation*}
	E_h^\partial \mu_h = \sum_{p\in\mathcal V_h(\Gamma_h)} \Big( \frac1{\# \mathcal E_h^\partial(p)} \sum_{e \in \mathcal E_h^\partial(p)} \mu_h(m_e) \Big) \bar\phi_p,
\end{equation*}
where $\mathcal V_h(\Gamma_h) = \mathcal V_h \cap \Gamma_h$, $\mathcal E_h^\partial(p) = \{e \in \mathcal E_h^\partial \,:\, p \in e\}$ means the boundary elements sharing the vertex $p$, and $\bar\phi_p \in \bar V_h$ is a nodal basis of the conforming P1 functions given in \sref{sec2}.
Note that, as a result of the regularity of meshes, the number of elements $\# \mathcal E_h^\partial(p)$ is bounded independently of $p$ and $h$.
Moreover, $\mathcal E_h^\partial(e) = \{ e' \in \mathcal E_h^\partial \,:\, e\cap e' \neq \emptyset \}$ denotes the neighboring boundary edges around $e$.

The discrete $H^{1/2}$-norm is compatible with the usual $H^{1/2}$-norm as follows.
\begin{lem} \label{lem: compatibility with usual H1/2 norm}
	If $\mu \in H^{1/2}(\Gamma_h)$, then
	\begin{equation*}
		\|\Pi_h^\partial\mu\|_{1/2,\Lambda_h} \le C \|\mu\|_{H^{1/2}(\Gamma_h)}.
	\end{equation*}
\end{lem}

We also state discrete $H^{1/2}$-stability when $n_h$ is involved.
\begin{lem} \label{lem: compatibility with usual H1/2 norm when nh is involved}
	Let $\mu \in H^{1/2}(\Gamma_h)$, $v \in H^{1/2}(\Gamma_h)^N$, and $A \in H^{1/2}(\Gamma_h)^{N^2}$ be scalar, vector, and matrix functions respectively.
	Then we have
	\begin{align*}
		\|(\Pi_h^\partial \mu) n_h\|_{1/2, \Lambda_h} &\le C \|\mu\|_{H^{1/2}(\Gamma_h)}, \\
		\|(\Pi_h^\partial v) \cdot n_h\|_{1/2, \Lambda_h} &\le C \|v\|_{H^{1/2}(\Gamma_h)}, \\
		\|(\Pi_h^\partial A) n_h \cdot n_h\|_{1/2, \Lambda_h} &\le C \|A\|_{H^{1/2}(\Gamma_h)}.
	\end{align*}
\end{lem}
The proofs of Lemmas \ref{lem: compatibility with usual H1/2 norm} and \ref{lem: compatibility with usual H1/2 norm when nh is involved} will be given in Appendices \ref{sec: A.1} and \ref{sec: A.2}, respectively.

\subsection{Discrete lifting theorems with respect to the normal component}
Let us state a first version of discrete lifting theorems.
\begin{lem}
	For all $\mu_h \in \Lambda_h$ we obtain
	\begin{equation} \label{eq1: discrete inf-sup for c}
		C\Big( \sum_{e\in\mathcal E_h^\partial} h_e \|\mu_h\|_{L^2(e)}^2 \Big)^{1/2} \le \sup_{v_h\in V_h} \frac{c_h(\mu_h, v_h\cdot n_h)}{\|v_h\|_{V_h}}.
	\end{equation}
\end{lem}
\begin{proof}
	Define $v_h \in V_h$ by $v_h = \sum_{e\in\mathcal E_h^\partial} h_e \mu_h(m_e) n_h(m_e) \phi_e$.
	Then we see that $c_h(\mu_h, v_h\cdot n_h) = \sum_{e \in \mathcal E_h^\partial} h_e \|\mu_h\|_{L^2(e)}^2$ and that
	\begin{align*}
		\|v_h\|_{V_h}^2 &= \sum_{e\in\mathcal E_h^\partial} h_e^2 |\mu_h(m_e) n_h(m_e)|^2 \|\phi_e\|_{H^1(T_e)}^2 \le \sum_{e\in\mathcal E_h^\partial} h_e^2 \|\mu_h\|_{L^\infty(e)}^2 \times C h_e^{N-2} \le C \sum_{e\in\mathcal E_h^\partial} h_e \|\mu_h\|_{L^2(e)}^2,
	\end{align*}
	where we have used a local inverse inequality $\|\mu_h\|_{L^\infty(e)} \le Ch_e^{(1-N)/2} \|\mu_h\|_{L^2(e)}$.
	Combining the two relations, we obtain the desired inf-sup condition.
\end{proof}

We need a more refined discrete lifting theorem than the one above.
\begin{lem} \label{lem: discrete lifting theorem}
	For $\mu_h \in \Lambda_h$ there exists $v_h \in V_h$ satisfying $(v_h \cdot n_h)(m_e) = \mu_h(m_e)$ for all $e \in \mathcal E_h^\partial$, together with the stability estimate
	\begin{equation} \label{eq: discrete lift stability}
		\|v_h\|_{V_h} \le C\|\mu_h\|_{1/2, \Lambda_h}.
	\end{equation}
\end{lem}
The proof of this lemma will be given in Appendix \ref{sec: A.3}.

\begin{cor}
	For all $\mu_h \in \Lambda_h$ we obtain
	\begin{equation} \label{eq2: discrete inf-sup for c}
		C\|\mu_h\|_{-1/2, \Lambda_h} \le \sup_{v_h \in V_h} \frac{c_h(\mu_h, v_h\cdot n_h)}{\|v_h\|_{V_h}}.
	\end{equation}
\end{cor}
\begin{proof}
	By the definition of the dual norm, there exists $\lambda_h\in\Lambda_h$ such that  $\|\mu_h\|_{-1/2,\Lambda_h} = \frac{c_h(\mu_h, \lambda_h)}{\|\lambda_h\|_{1/2,\Lambda_h}}$.
	We apply \lref{lem: discrete lifting theorem} to $\lambda_h$ to obtain some $v_h \in V_h$ such that $v_h \cdot n_h = \lambda_h$ at all $m_e$'s lying on $\Gamma_h$ and $\|v_h\|_{V_h} \le C\|\lambda_h\|_{1/2,\Lambda_h}$.
	It is now immediate to deduce \eref{eq2: discrete inf-sup for c}.
\end{proof}

\subsection{Estimates on the boundary-skin layer}
Let us introduce a tubular neighborhood of $\Gamma$ with width $\delta > 0$ by $\Gamma(\delta) = \{x\in\mathbb R^N \,:\, \operatorname{dist}(x,\Gamma) < \delta\}$.
For sufficiently small $\delta_0>0$, we know that (see \cite[Section 14.6]{GiTr98}) there holds a unique decomposition $\Gamma(\delta_0) \ni x = \bar x + t n(\bar x)$ with $\bar x \in \Gamma$.
The maps $\pi: \Gamma(\delta_0) \to \Gamma$; $x\mapsto \bar x$ and $d: \Gamma(\delta_0) \to \mathbb R$; $x\mapsto t$ imply the orthogonal projection to $\Gamma$ and the signed-distance function, respectively.
We fix a bounded smooth domain $\tilde\Omega$ that contains $\Omega \cup \Gamma(\delta_0)$.

If the mesh size $h$ is sufficiently small, we proved in \cite[Section 8]{KOZ16} that $\pi|_{\Gamma_h} : \Gamma_h \to \Gamma$ is a homeomorphism and that $|d(x)| \le Ch_e^2 =: \delta_e$ for $x \in e \in \mathcal E_h^\partial$.
Then the following \emph{boundary-skin estimates} are obtained:
\begin{align}
	\Big| \int_{\pi(e)} f \, ds - \int_e f\circ\pi \, ds \Big| &\le C\delta_e \|f\|_{L^1(e)}, && f \in L^1(e), \label{eq: boundary-skin estimate 1} \\
	\|f - f\circ\pi\|_{L^p(e)} &\le C\delta_e^{1-1/p} \|\nabla f\|_{L^p(\pi(e, \delta_e))}, && f \in W^{1,p}(\pi(e, \delta_e)), \label{eq: boundary-skin estimate 2} \\
	\|f\|_{L^p(\pi(e, \delta_e))} &\le C\delta_e^{1/p} \|f\|_{L^p(\pi(e))} + C\delta_e \|\nabla f\|_{L^p(\pi(e, \delta_e))}, && f \in W^{1,p}(\pi(e, \delta_e)), \label{eq: boundary-skin estimate 3}
\end{align}
where $p \in [1,\infty]$ and $\pi(e, \delta_e) := \{ \bar x + t n(\bar x) \in \mathbb R^2 \,:\, \bar x \in \pi(e), \; |t| < \delta_e \}$ denotes a tubular neighborhood of $\pi(e) \subset \Gamma$.
As a version of \eref{eq: boundary-skin estimate 3}, we also have (see \cite[Lemma A.1]{KaKe18})
\begin{equation*}
	\|f\|_{L^p( (\Omega_h\setminus\Omega) \cap \pi(e,\delta_e) )} \le C\delta_e^{1/p} \|f\|_{L^p(e)} + C\delta_e \|\nabla f\|_{L^p( (\Omega_h\setminus\Omega) \cap \pi(e,\delta_e) )}.
\end{equation*}
Adding up the estimates above for $e \in \mathcal E_h^\partial$, we obtain corresponding global estimates on boundary-skin layers.
In particular one has
\begin{equation} \label{eq: corollary of boundary-skin estimate 3}
	\|v\|_{L^2(\Omega_h \setminus \Omega)} \le Ch \|v\|_{H^1(\Omega_h)} \quad \forall v \in H^1(\Omega_h).
\end{equation}
Here we present its version in case of a nonconforming approximation.
For the proof, see \sref{sec: A.4}.
\begin{lem} \label{lem: a boundary-skin estimate for nonconforming approximation}
	For all $v \in V_h + H^1(\Omega_h)^N$ we obtain
	\begin{equation*}
		\|v\|_{L^2(\Omega_h \setminus \Omega)} \le Ch \vertiii{v}_{V_h}.
	\end{equation*}
\end{lem}

\subsection{Interpolation estimates for $u\cdot n = g$}
Although the approximability of $n_h$ to $n$ is only $O(h)$ on $\Gamma_h$, at the midpoints of edges it is improved to $O(h^2)$ for $N=2$ as result of super-convergence.
This was a key observations in \cite{KOZ16} to deal with errors caused by discretization of $u\cdot n = g$; this idea, however, demanded the assumption of the $W^{2,\infty}$-regularity for velocity $u$.
Here we present a different approach which only requires $u \in H^2(\Omega)^N$, taking advantage of the divergence-free condition.
\begin{lem} \label{lem: error estimate for flux}
	Let $u \in H^2(\Omega)^N$ satisfy $\mathrm{div}\, u = 0$.
	Then for $e\in\mathcal E_h^\partial$ we have
	\begin{equation*}
		\left| \int_e u\cdot n_h\, ds - \int_{\pi(e)} u\cdot n\, ds \right| \le
		\begin{cases}
			Ch_e^{9/2} \|\nabla^2\tilde u\|_{L^2(\pi(e, \delta_e))} & \quad\text{if}\quad N=2, \\
			Ch_e^3 \|\tilde u\|_{H^2(\tilde\Omega)} & \quad\text{if}\quad N=3.
		\end{cases}
	\end{equation*}
\end{lem}
\begin{proof}
	We set $D := \pi(e, \delta_e) \cap \Omega$, $D_h := \pi(e, \delta_e) \cap \Omega_h$, and introduce ``reminder boundaries'' of $D$ and $D_h$ by $R = \partial D \setminus \pi(e)$ and $R_h = \partial D_h \setminus e$.
	Then it follows from the divergence theorem that
	\begin{align*}
		\int_e \tilde u\cdot n_h\, ds - \int_{\pi(e)} u\cdot n\, ds = \int_{D_h \setminus D} \mathrm{div}\, \tilde u\, dx -  \Big( \int_{R_h} \tilde u\cdot\nu_h\, ds - \int_R \tilde u\cdot\nu\, ds \Big) =: I_1 + I_2,
	\end{align*}
	where $\nu$ and $\nu_h$ denote the outer unit normals to $R$ and $R_h$, respectively.
	
	When $N=2$, $I_2 = 0$ since $R_h = R$.
	By \eref{eq: boundary-skin estimate 3}, we have (note that $\operatorname{div} u = 0$ on $\Gamma$)
	\begin{equation*}
		|I_1| \le \|\operatorname{div}\tilde u\|_{L^1(\pi(e, \delta_e))} \le C\delta_e \|\nabla \operatorname{div}\tilde u\|_{L^1(\pi(e, \delta_e))} \le C\delta_e |\pi(e, \delta_e)|^{1/2} \|\nabla \operatorname{div}\tilde u\|_{L^2(\pi(e, \delta_e))},
	\end{equation*}
	which combined with $|\pi(e, \delta_e)|^{1/2} \le Ch_e^{N-1}\delta_e$ implies the desired estimate.
	
	When $N=3$, denoting by $L_e = \{ \bar x + t n(\bar x) \,:\, \bar x \in \partial\pi(e), \; |t| \le \delta_e \}$ the lateral boundary of $\pi(e, \delta_e)$, we obtain
	\begin{equation*}
		|I_2| \le |L_e| \|\tilde u\|_{L^\infty(\tilde\Omega)} \le Ch_e\delta_e \|\tilde u\|_{H^2(\tilde\Omega)},
	\end{equation*}
	where we have used Sobolev's embedding theorem.
	Since the estimate of $I_2$ dominates that of $I_1$, the desired result follows.
\end{proof}
\begin{rem}
	If the extension satisfies $\operatorname{div}\tilde u = 0$ in $\tilde\Omega$, then the error becomes zero for $N = 2$.
\end{rem}

We apply the above lemma to estimate the error $\tilde u \cdot n_h - \tilde g$ on $\Gamma_h$.
\begin{lem} \label{lem: error estimate of u dot nh - g}
	Let $u \in H^2(\Omega)^N$ and $g \in H^{3/2}(\Gamma)$ satisfy $\mathrm{div}\, u = 0$ and $u\cdot n = g$.
	Then for $e \in \mathcal E_h^\partial$ we have
	\begin{equation*}
		\|\Pi_h^\partial (\tilde u\cdot n_h - \tilde g)\|_{L^2(e)}^2 \le
		\begin{cases}
			Ch_e^4 (\|\tilde g\|_{L^2(e)}^2 + \|\nabla \tilde g\|_{L^2(e)}^2 + \|\nabla^2 \tilde g\|_{L^2(\pi(e, \delta_e))}^2) + Ch_e^8 \|\nabla^2\tilde u\|_{L^2(\pi(e, \delta_e))}^2 &\hspace{-1mm} (N=2), \\
			Ch_e^4 (\|\tilde g\|_{L^2(e)}^2 + \|\nabla \tilde g\|_{L^2(e)}^2 + \|\nabla^2 \tilde g\|_{L^2(\pi(e, \delta_e))}^2) + Ch_e^4 \|\tilde u\|_{H^2(\tilde\Omega)}^2 &\hspace{-1mm} (N=3).
		\end{cases}
	\end{equation*}
\end{lem}
\begin{proof}
	Observe that
	\begin{align*}
		\|\Pi_h^\partial (\tilde u\cdot n_h - \tilde g)\|_{L^2(e)}^2 &= |e|^{-1} \left| \int_e \tilde u\cdot n_h\, ds - \int_e \tilde g\, ds \right|^2 \\
			&\le Ch_e^{1-N} \bigg( \Big| \int_e \tilde u\cdot n_h\, ds - \int_{\pi(e)} u\cdot n\, ds \Big|^2 + \Big| \int_{\pi(e)} g\, ds - \int_e \tilde g\, ds \Big|^2 \bigg).
	\end{align*}
	It follows from \eref{eq: boundary-skin estimate 1} and \eref{eq: boundary-skin estimate 2} that
	\begin{align*}
		\Big| \int_{\pi(e)} g\, ds - \int_e \tilde g\, ds \Big| &\le \Big| \int_{\pi(e)} g\, ds - \int_e g\circ\pi \, ds \Big| + \int_{e} |g\circ\pi - \tilde g| \, ds \le C\delta_e \|\tilde g\|_{L^1(e)} + C\|\nabla\tilde g\|_{L^1(\pi(e, \delta_e))} \\
			&\le C\delta_e (\|\tilde g\|_{L^1(e)} + \|\nabla\tilde g\|_{L^1(e)} + \|\nabla^2\tilde g\|_{L^1(\pi(e,\delta_e))}) \\
			&\le Ch_e^{(N+3)/2} (\|\tilde g\|_{L^2(e)} + \|\nabla \tilde g\|_{L^2(e)} + \|\nabla^2 \tilde g\|_{L^2(\pi(e, \delta_e))}).
	\end{align*}
	Combining these with \lref{lem: error estimate for flux}, we conclude the desired estimates.
\end{proof}
\begin{rem}
	If $(u, p)$ is a solution of \eref{eq: Stokes slip BC}, then adding up the result of \lref{lem: error estimate of u dot nh - g} for $e \in \mathcal E_h^\partial$ yields
	\begin{align}
		\|\Pi_h^\partial (\tilde u\cdot n_h - \tilde g)\|_{L^2(\Gamma_h)} &\le Ch^{2\alpha} \|u\|_{H^2(\Omega)}, \label{eq1: L2 error of u cdot nh - g} \\
		\Big( \sum_{e \in \mathcal E_h^\partial} h_e^{-1} \|\Pi_h^\partial (\tilde u\cdot n_h - \tilde g)\|_{L^2(e)}^2 \Big)^{1/2} &\le Ch^{2\alpha - 1/2} \|u\|_{H^2(\Omega)}, \label{eq2: L2 error of u cdot nh - g}
	\end{align}
	where $\alpha$ is the same as in \tref{thm: H1 error estimate} and we have used $\sum_{e \in \mathcal E_h^\partial} h_e^2 \le C$ in case $N=3$.
\end{rem}

\section{Well-posedness of the Approximate Problem} \label{sec4}
We adopt the following discrete version of Korn's inequality proved in \cite{Bre03} (see also \cite[p.\ 993]{BuHa05}):
\begin{equation} \label{eq: discrete Korn's inequality}
	C\|v_h\|_{V_h}^2 \le a_h(v_h, v_h) + j_h(v_h, v_h) \quad \forall v_h \in V_h.
\end{equation}
In addition, it is known that an inf-sup condition is valid for $b_h$ (see \cite[Section 8.4.4]{BBF13}):
\begin{equation} \label{eq: inf-sup for bh}
	C \|q_h\|_{Q_h} \le \sup_{v_h \in \mathring V_h} \frac{b_h(q_h, v_h)}{\|v_h\|_{V_h}} \quad \forall q_h \in \mathring Q_h.
\end{equation}
\begin{rem}
	The positive constants $C$ appearing above depend on the $C^{0,1}$-regularity of the domain $\Omega_h$, which is independent of $h$ if it is sufficiently small.
\end{rem}
\begin{proof}[Proof of \tref{thm: well-posedness}]
	Because the problem is linear and finite-dimensional, it suffices to show the \emph{a priori} estimate \eref{eq: stability of uh and ph} assuming the existence of a solution $(u_h, p_h)$ of \eref{eq: discrete 2-field formulation}.
	Since $v_h$ vanishes at $m_e$'s on $\Gamma_h$ and $b_h(1, v_h) = 0$ for $v_h \in \mathring V_h$, it follows from the inf-sup condition \eref{eq: inf-sup for bh} that
	\begin{align*}
		C \|\mathring p_h\|_{Q_h} &\le \sup_{v_h \in \mathring V_h} \frac{b_h(\mathring p_h, v_h)}{\|v_h\|_{V_h}}
			= \sup_{v_h \in \mathring V_h} \frac{ (\tilde f, v_h)_{\Omega_h} + (\tilde\tau, v_h)_{\Gamma_h} - a_h(u_h, v_h) - j_h(u_h, v_h) }{\|v_h\|_{V_h}} \\
			&\le C (\|\tilde f\|_{L^2(\Omega_h)} + \|\tilde\tau\|_{L^2(\Gamma_h)} + \|u_h\|_{V_h}).
	\end{align*}
	
	Next, by Lemmas \ref{lem: discrete lifting theorem} and \ref{lem: compatibility with usual H1/2 norm when nh is involved} there exists $w_h \in V_h$ such that $w_h \cdot n_h = -1$ at $m_e$'s on $\Gamma_h$ and $\|w_h\|_{V_h} \le C$.
	Taking $v_h = w_h$ in \eref{eq: discrete 2-field formulation}$_1$ and noting that $b_h(1, w_h) = -(1, w_h\cdot n_h)_{\Gamma_h} = |\Gamma_h|$, we obtain
	\begin{equation*}
		k_h |\Gamma_h| = (\tilde f, w_h)_{\Omega_h} + (\tilde\tau, w_h)_{\Gamma_h} - a_h(u_h, w_h) - b_h(\mathring p_h, w_h) - \frac1\epsilon c_h(u_h\cdot n_h - \tilde g, 1) - j_h(u_h, w_h).
	\end{equation*}
	This, together with $c_h(u_h\cdot n_h, 1) = -b_h(1, u_h) = 0$, gives an estimate for $k_h$:
	\begin{equation*}
		|k_h| \le C \big( \|\tilde f\|_{L^2(\Omega_h)} + \|\tilde\tau\|_{L^2(\Gamma_h)} + \|u_h\|_{V_h} + \|\mathring p_h\|_{Q_h} + \frac1\epsilon |c_h(\tilde g, 1)| \big),
	\end{equation*}
	where, by the definition of $\Pi_h^\partial$, by the compatibility condition \eref{eq: compatibility of g} and by \eref{eq: boundary-skin estimate 1}--\eref{eq: boundary-skin estimate 2}, we have
	\begin{equation*}
		|c_h(\tilde g, 1)| = |(\Pi_h^\partial \tilde g, 1)_{\Gamma_h}| = \left| \int_{\Gamma_h} \tilde g\, ds - \int_\Gamma g\, ds \right| \le Ch^2 \|\tilde g\|_{H^2(\tilde\Omega)}.
	\end{equation*}
	In conclusion, the pressure can be estimated as
	\begin{equation} \label{eq: stability ph}
		\|p_h\|_{Q_h} \le C (\|\tilde f\|_{L^2(\Omega_h)} + \|\tilde\tau\|_{L^2(\Gamma_h)} + \epsilon^{-1}h^2 \|\tilde g\|_{H^2(\tilde\Omega)} + \|u_h\|_{V_h}).
	\end{equation}
	
	Finally, making use of the discrete Korn's inequality \eref{eq: discrete Korn's inequality} and taking $v_h = u_h$ in \eref{eq: discrete 2-field formulation}$_1$ give
	\begin{align}
		C\|u_h\|_{V_h}^2 &\le a_h(u_h, u_h) + j_h(u_h, u_h) + \frac1\epsilon \|u_h\cdot n_h - \Pi_h^\partial \tilde g\|_{L^2(\Gamma_h)}^2 \notag \\
			&= (\tilde f, u_h)_{\Omega_h} + (\tilde\tau, u_h)_{\Gamma_h} - \frac1\epsilon c_h(u_h\cdot n_h - \tilde g, \tilde g). \label{eq: stability uh}
	\end{align}
	To address the third term on the last line, we find from Lemmas \ref{lem: discrete lifting theorem} and \ref{lem: compatibility with usual H1/2 norm when nh is involved} some $z_h \in V_h$ such that $z_h\cdot n_h = \Pi_h^\partial \tilde g$ at $m_e$'s on $\Gamma_h$ and $\|z_h\|_{V_h} \le C \|\tilde g\|_{H^{1/2}(\Gamma_h)}$.
	Letting now $v_h = z_h$ in \eref{eq: discrete 2-field formulation}$_1$ one gets
	\begin{align}
		\big| \frac1\epsilon c_h(u_h\cdot n_h - \Pi_h^\partial\tilde g, \Pi_h^\partial\tilde g) \big| &= \big| (\tilde f, z_h)_{\Omega_h} + (\tilde\tau, z_h)_{\Gamma_h} - a_h(u_h, z_h) - b_h(p_h, z_h) \big| \notag \\
			&\le C ( \|\tilde f\|_{L^2(\Omega_h)} + \|\tilde\tau\|_{L^2(\Gamma_h)} + \|u_h\|_{V_h} + \|p_h\|_{Q_h} ) \|\tilde g\|_{H^{1/2}(\Gamma_h)}. \label{eq: stability of uh, reminder term}
	\end{align}
	Combining the estimates \eref{eq: stability ph}--\eref{eq: stability of uh, reminder term}, performing an absorbing argument, and using the stability of extensions, we conclude \eref{eq: stability of uh and ph}.
\end{proof}

\section{$H^1$-error estimate} \label{sec5}
Let us introduce a discrete Lagrange multiplier by $\lambda_h := \frac1\epsilon \Pi_h^\partial (u_h \cdot n_h - \tilde g) \in \Lambda_h$.
An easy but important fact is that if $(u_h, p_h)$ solves \eref{eq: discrete 2-field formulation}, then $(u_h, p_h, \lambda_h)$ satisfies the following three-field formulation:
\begin{equation} \label{eq: discrete 3-field formulation}
	\left\{
	\begin{aligned}
		a_h(u_h, v_h) + b_h(p_h, v_h) + c_h(\lambda_h, v_h\cdot n_h) + j_h(u_h, v_h) &= (\tilde f, v_h)_{\Omega_h} + (\tilde \tau, v_h)_{\Gamma_h} && \forall v_h \in V_h, \\
		b_h(q_h, u_h) &= 0 && \forall q_h \in Q_h, \\
		c_h(\mu_h, u_h\cdot n_h - \tilde g) &= \epsilon c_h(\mu_h, \lambda_h) && \forall \mu_h \in \Lambda_h,
	\end{aligned}
	\right.
\end{equation}
which will be compared with \eref{eq: 3-field formulation} in the subsequent arguments.
\subsection{Consistency error estimate}
Since $\Omega\neq\Omega_h$ and a nonconforming element is employed, the consistency (i.e.\ the Galerkin orthogonality relation) does not hold exactly.
However, it is still valid in an asymptotic sense with respect to $h\to0$.
To see this, we introduce a functional $\mathrm{Res}(v)$ by
\begin{align}
	\mathrm{Res}(v) &:= (\tilde u - \nu\Delta\tilde u - \nu\nabla\operatorname{div}\tilde u + \nabla\tilde p - \tilde f, v)_{\Omega_h \setminus \Omega} + \sum_{e \in \mathring{\mathcal E}_h} (\sigma(\tilde u, \tilde p)n_e, \jump{v})_e \notag \\
		&\hspace{1cm} + (\sigma(\tilde u, \tilde p)n_h, v)_{\Gamma_h} - (\tilde\tau - \tilde\lambda n_h, v)_{\Gamma_h} + (\tilde\lambda, (\Pi_h^\partial v - v)\cdot n_h)_{\Gamma_h}, \label{eq: definition of Res(v)}
\end{align}
which is well-defined for $v \in H^1(\mathcal T_h)^N$.
The next lemma shows that $\mathrm{Res}(v)$ describes the residual of the consistency and that it is of $O(h)$.

\begin{lem} \label{lem: asymptotic Galerkin orthogonality}
	Let $(u, p, \lambda) \in H^2(\Omega)^N \times H^1(\Omega) \times H^{1/2}(\Gamma)$ be the solution of \eref{eq: 3-field formulation} and $(u_h, p_h, \lambda_h) \in V_h \times Q_h \times \Lambda_h$ be that of \eref{eq: discrete 2-field formulation}.
	
	(i) For $v_h \in V_h$ we have
	\begin{equation} \label{eq: Galerkin orthogonality}
		a_h(\tilde u - u_h, v_h) + b_h(\tilde p - p_h, v_h) + c_h(\tilde\lambda - \lambda_h, v_h \cdot n_h) - j_h(u_h, v_h) = \mathrm{Res}(v_h).
	\end{equation}
	
	(ii) For $v \in V_h + H^1(\Omega_h)^N$ we obtain
	\begin{equation*}
		|\mathrm{Res}(v)| \le Ch(\|u\|_{H^2(\Omega)} + \|p\|_{H^1(\Omega)}) \vertiii{v}_{V_h}.
	\end{equation*}
\end{lem}
\begin{rem}
	(i) As an easy consequence of \eref{eq: Galerkin orthogonality} we have
	\begin{align*}
		a_h(\tilde u - u_h, v_h) + b_h(\tilde p - p_h, v_h)  - j_h(u_h, v_h) = \mathrm{Res}(v_h) \qquad \forall v_h \in \mathring V_h.
	\end{align*}
	
	(ii) Noting that $b_h(k_h, v_h) + c_h(k_h, v_h\cdot n_h) = 0$, where $k_h$ is given in \tref{thm: well-posedness}, one has
	\begin{equation*}
		a_h(\tilde u - u_h, v_h) + b_h(\tilde p + k_h - p_h, v_h) + c_h(\tilde\lambda + k_h - \lambda_h, v_h \cdot n_h) - j_h(u_h, v_h) = \mathrm{Res}(v_h).
	\end{equation*}
	Since $R_h\tilde p - \tilde p$ and $\Pi_h^\partial\tilde\lambda - \tilde\lambda$ are orthogonal to the functions in $Q_h$ and to those in $\Lambda_h$ respectively, this in particular implies
	\begin{equation*}
		a_h(\tilde u - u_h, v_h) + b_h( R_h(\tilde p + k_h) - p_h, v_h) + c_h( \Pi_h^\partial (\tilde\lambda + k_h) - \lambda_h, v_h \cdot n_h) - j_h(u_h, v_h) = \mathrm{Res}(v_h).
	\end{equation*}
\end{rem}
\begin{proof}[Proof of \lref{lem: asymptotic Galerkin orthogonality}]
	(i) Integration by parts together with \eref{eq: discrete 3-field formulation}$_1$ shows that the left-hand side of \eref{eq: Galerkin orthogonality} equals
	\begin{align*}
		& \sum_{T \in \mathcal T_h} \Big( (\tilde u - \nu\Delta\tilde u - \nu\nabla\operatorname{div}\tilde u + \nabla\tilde p, v_h)_{T} + (\sigma(\tilde u, \tilde p)n_{\partial T}, v_h)_{\partial T} \Big) + c_h(\tilde\lambda, v_h\cdot n_h) - (\tilde f, v_h)_{\Omega_h} - (\tilde\tau, v_h)_{\Gamma_h} \\
		=\; &(\tilde u - \nu\Delta\tilde u - \nu\nabla\operatorname{div}\tilde u + \nabla\tilde p - \tilde f, v_h)_{\Omega_h \setminus \Omega} + \sum_{e \in \mathring{\mathcal E}_h} (\sigma(\tilde u, \tilde p)n_e, \jump{v_h})_e + (\sigma(\tilde u, \tilde p)n_h, v_h)_{\Gamma_h} \\
		&\hspace{1cm} - (\tilde\tau, v_h)_{\Gamma_h} + (\tilde\lambda, \Pi_h^\partial v_h\cdot n_h)_{\Gamma_h}.
	\end{align*}
	Since $-(-\tilde\lambda n_h, v_h)_{\Gamma_h} + (\tilde\lambda, -v_h\cdot n_h)_{\Gamma_h} = 0$, this implies \eref{eq: Galerkin orthogonality}.
	
	(ii) For simplicity we abbreviate $C(\|u\|_{H^2(\Omega)} + \|p\|_{H^1(\Omega)})$ as $C(u, p)$.
	The first term of $\mathrm{Res}(v)$ is bounded by
	\begin{equation*}
		|(\tilde u - \nu\Delta\tilde u - \nu\nabla\operatorname{div}\tilde u + \nabla\tilde p - \tilde f, v)_{\Omega_h\setminus\Omega}| \le C(u, p) \|v\|_{L^2(\Omega_h\setminus\Omega)} \le C(u, p)h \vertiii{v}_{V_h},
	\end{equation*}
	where we have used \eref{eq: boundary-skin estimate 3}.  For the second term, since $\int_e\jump{v}\,ds = 0$ for $e \in \mathring{\mathcal E}_h$, we have
	\begin{align*}
		\bigg| \sum_{e\in\mathring{\mathcal E}_h} \big( \sigma(\tilde u, \tilde p)n_e, \jump{v} \big)_e \bigg| &= \bigg| \sum_{e\in\mathring{\mathcal E}_h} \big( (\sigma(\tilde u, \tilde p) - \Pi^e \sigma(\tilde u, \tilde p) )n_e, \jump{v} \big)_e \bigg| \\
			&\le \bigg( \sum_{e\in\mathring{\mathcal E}_h} h_e \|\sigma(\tilde u, \tilde p) - \Pi^e \sigma(\tilde u, \tilde p)\|_{L^2(e)}^2 \bigg)^{1/2} \bigg( \sum_{e\in\mathring{\mathcal E}_h} h_e^{-1} \|\jump{v}\|_{L^2(e)}^2 \bigg)^{1/2} \\
			&\le C\bigg( \sum_{e\in\mathring{\mathcal E}_h} h_e^2 \|\sigma(\tilde u, \tilde p)\|_{H^{1/2}(e)}^2 \bigg)^{1/2} \bigg( \sum_{e\in\mathring{\mathcal E}_h} h_e^{-1} \|\jump{v}\|_{L^2(e)}^2 \bigg)^{1/2} \\
			&\le C(u, p)h \vertiii{v}_{V_h},
	\end{align*}
	where $\Pi^e$ denotes the orthogonal projector from $L^2(e)$ onto $P_0(e)$.
	To address the third and forth terms we observe that
	\begin{equation*}
		\sigma(\tilde u, \tilde p)n_h - \tilde \tau + \tilde\lambda n_h = \big( \sigma(\tilde u, \tilde p)(I - n_h\otimes n_h)n_h - \tilde\tau \big) + \big( \sigma(\tilde u, \tilde p)n_h\cdot n_h + \tilde\lambda \big) n_h =: F_\tau + F_n.
	\end{equation*}
	Recalling that $\sigma(u, p)(I - n\otimes n)n = \tau$ on $\Gamma$, one has
	\begin{equation*}
		F_\tau = \sigma(\tilde u, \tilde p)(I - n_h\otimes n_h)n_h - \big( \sigma(u, p)(I - n\otimes n)n \big)\circ\pi + \tau\circ\pi - \tilde\tau,
	\end{equation*}
	which, combined with the estimates
	\begin{equation*}
		\|\sigma(\tilde u, \tilde p) - \sigma(u, p)\circ\pi\|_{L^2(\Gamma_h)} \le C(u, p)h, \quad \|\tilde\tau - \tau\circ\pi\|_{L^2(\Gamma_h)} \le C(u, p)h, \quad \|n_h - n\circ\pi\|_{L^\infty(\Gamma_h)} \le Ch,
	\end{equation*}
	yields $\|F_\tau\|_{L^2(\Gamma_h)} \le C(u, p)h$.  Similarly we have $\|F_n\|_{L^2(\Gamma_h)} \le C(u, p)h$.  Therefore,
	\begin{equation*}
		|(F_\tau + F_n, v)_{\Gamma_h}| \le C(u, p)h \|v\|_{L^2(\Gamma_h)} \le C(u, p)h \|v\|_{V_h}.
	\end{equation*}
	Finally, the last term of $\mathrm{Res}(v)$ is estimated by
	\begin{align*}
		\big| \big( \tilde\lambda, (\Pi_h^\partial v - v)\cdot n_h \big)_{\Gamma_h} \big| &= \Big| \sum_{e\in\mathcal E_h^\partial} \big( \tilde\lambda - \Pi_h^\partial\tilde\lambda, (\Pi_h^\partial v - v)\cdot n_h \big) \Big| \\
			&\le \Big( \sum_{e\in\mathcal E_h^\partial} \|\tilde\lambda - \Pi_h^\partial\tilde\lambda\|_{L^2(e)}^2 \Big)^{1/2} \Big( \sum_{e\in\mathcal E_h^\partial} \|\Pi_h^\partial v - v\|_{L^2(e)}^2 \Big)^{1/2} \\
			&\le C \Big( \sum_{e\in\mathcal E_h^\partial} h_e\|\tilde\lambda\|_{H^{1/2}(e)}^2 \Big)^{1/2} \Big( \sum_{e\in\mathcal E_h^\partial} h_e\|v\|_{H^{1/2}(e)}^2 \Big)^{1/2} \\
			&\le C(u, p) h\|v\|_{V_h}.
	\end{align*}
	Collecting the estimates above concludes $|\mathrm{Res}(v)| \le C(u, p)h \vertiii{v}_{V_h}$.
\end{proof}

\subsection{Proof of \tref{thm: H1 error estimate}}
In view of the regularity property, stability of extension operators and interpolation estimates, it suffices to prove that
\begin{equation} \label{eq: goal of H1 error estimate}
	\vertiii{\Pi_h\tilde u - u_h}_{V_h} + \|R_h \tilde p - \mathring p_h\|_{Q_h} \le C(h+\epsilon) (\|\tilde u\|_{H^2(\tilde\Omega)} + \|\tilde p\|_{H^1(\tilde\Omega)}).
\end{equation}
In what follows, we abbreviate the quantity $C(\|\tilde u\|_{H^2(\tilde\Omega)} + \|\tilde p\|_{H^1(\tilde\Omega)})$ just as $C(u,p)$, and we set $v_h := \Pi_h\tilde u$, $q_h := R_h\tilde p + k_h = R_h(\tilde p + k_h)$, and $\mu_h := \Pi_h^\partial  \tilde\lambda + k_h = \Pi_h^\partial  (\tilde\lambda + k_h)$, where $k_h$ is given in \tref{thm: well-posedness}.

We start from Korn's inequality \eref{eq: discrete Korn's inequality} and \eref{eq: Galerkin orthogonality} to find that
\begin{align}
	C \vertiii{v_h - u_h}_{V_h}^2 &\le a_h(v_h - u_h, v_h - u_h) + j_h(v_h - u_h, v_h - u_h) \notag \\
	&= a_h(\tilde u - u_h, v_h - u_h) - j_h(u_h, v_h - u_h) + a_h(v_h - \tilde u, v_h - u_h) + j_h(v_h - \tilde u, v_h - u_h) \notag \\
	&= \mathrm{Res}(v_h - u_h) - b_h(q_h - p_h, v_h - u_h) - c_h(\mu_h - \lambda_h, (v_h - u_h)\cdot n_h) \notag \\
	&\hspace{2cm} + a_h(v_h - \tilde u, v_h - u_h) + j_h(v_h - \tilde u, v_h - u_h) \notag \\
	&=: I_1 + I_2 + I_3 + I_4 + I_5. \notag
\end{align}
By \lref{lem: asymptotic Galerkin orthogonality} and by the boundedness of $a_h$ and $j_h$, one has $|I_1 + I_4 + I_5| \le C(u,p)h \vertiii{v_h - u_h}_{V_h}$.
For $I_2$, since $\operatorname{div} u = 0$ in $\Omega$, it follows that
\begin{align*}
	I_2 &= b_h(q_h - p_h, v_h - \tilde u) + b_h(q_h - p_h, \tilde u) \\
		&\le \sum_{T \in \mathcal T_h} \|q_h - p_h\|_{L^2(T)} Ch_T \|\nabla \tilde u\|_{L^2(T)} + |(q_h - p_h, \operatorname{div}\tilde u)_{\Omega_h\setminus\Omega}| \\
		&\le C(u,p)h \|q_h - p_h\|_{L^2(\Omega_h)}.
\end{align*}
For $I_3$, it follows from \eref{eq2: L2 error of u cdot nh - g} that
\begin{align}
	I_3 &= -c_h(\mu_h - \lambda_h, \tilde u\cdot n_h - \tilde g) + \epsilon c_h(\mu_h - \lambda_h, \mu_h) - \epsilon \|\mu_h - \lambda_h\|_{L^2(\Gamma_h)}^2 \notag \\
		&\le \sum_{e\in\mathcal E_h^\partial} \|\mu_h - \lambda_h\|_{L^2(e)} \|\Pi_h^\partial (\tilde u\cdot n_h - \tilde g)\|_{L^2(e)} + \epsilon \|\mu_h - \lambda_h\|_{-1/2,\Lambda_h} \|\mu_h\|_{1/2,\Lambda_h} \notag \\
		&\le  C(u,p)h^{2\alpha - 1/2} \Big( \sum_{e\in\mathcal E_h^\partial} h_e \|\mu_h - \lambda_h\|_{L^2(e)}^2 \Big)^{1/2} + C(u,p)(\epsilon + h^2) \|\mu_h - \lambda_h\|_{-1/2,\Lambda_h}, \label{eq: H1 error proof 2}
\end{align}
where we have estimated $\mu_h$, using \lref{lem: compatibility with usual H1/2 norm} and \eref{eq: stability of kh}, by
\begin{equation*}
	\|\mu_h\|_{1/2, \Lambda_h} \le \|\Pi_h^\partial \tilde\lambda\|_{1/2, \Lambda_h} + C|k_h| \le C(u,p)(1 + h^2\epsilon^{-1}).
\end{equation*}
The errors for $\mu_h - \lambda_h$ in \eref{eq: H1 error proof 2} are bounded by the use of \eref{eq1: discrete inf-sup for c} and \eref{eq2: discrete inf-sup for c} as
\begin{align}
	&C \Big( \sum_{e\in\mathcal E_h^\partial} h_e \|\mu_h - \lambda_h\|_{L^2(e)}^2 \Big)^{1/2} + C \|\mu_h - \lambda_h\|_{-\frac12,\Lambda_h} \notag \\
		\le\; &\sup_{v_h \in V_h} \frac{c_h(\mu_h - \lambda_h, v_h \cdot n_h)}{\|v_h\|_{V_h}} 
		= \sup_{v_h \in V_h} \frac{\mathrm{Res}(v_h) - a_h(\tilde u - u_h, v_h) - b_h(q_h - p_h, v_h) + j_h(u_h, v_h)}{\|v_h\|_{V_h}} \notag \\
		\le\; &C(u,p) h + C\vertiii{v_h - u_h}_{V_h} + C \|q_h - p_h\|_{Q_h}. \label{eq: H1 error proof 3}
\end{align}

To estimate $\|q_h - p_h\|_{Q_h}$, we notice that
\begin{align*}
	q_h - p_h = R_h\tilde p - \mathring p_h = R_h\mathring{\tilde p} - \mathring p_h + \frac1{|\Omega_h|}(\tilde p, 1)_{\Omega_h},
\end{align*}
where the relation $(p, 1)_{\Omega} = 0$ combined with \eref{eq: boundary-skin estimate 3} gives
\begin{equation*}
	|(\tilde p, 1)_{\Omega_h}| = |(\tilde p, 1)_{\Omega_h\setminus\Omega} - (p, 1)_{\Omega\setminus\Omega_h}| \le \|\tilde p\|_{L^1(\Gamma(\delta))} \le Ch^2 \|\tilde p\|_{W^{1,1}(\tilde\Omega)}.
\end{equation*}
On the other hand, by the inf-sup condition \eref{eq: inf-sup for bh},
\begin{align}
	C\|R_h\mathring{\tilde p} - \mathring p_h\|_{Q_h} &\le \sup_{v_h\in\mathring V_h} \frac{b_h(R_h\mathring{\tilde p} - \mathring p_h, v_h)}{\|v_h\|_{V_h}} = \sup_{v_h\in\mathring V_h} \frac{b_h(\tilde p -  p_h, v_h)}{\|v_h\|_{V_h}} \notag \\
		&= \sup_{v_h\in\mathring V_h} \frac{\mathrm{Res}(v_h) - a_h(\tilde u - u_h, v_h) + j_h(u_h, v_h)}{\|v_h\|_{V_h}} \notag \\
		&\le C(u,p) h + C\vertiii{v_h - u_h}_{V_h}. \label{eq: H1 error proof 4}
\end{align}
Therefore, we obtain $\|q_h - p_h\|_{Q_h} \le C(u,p) h + C\vertiii{v_h - u_h}_{V_h}$, which concludes
\begin{align*}
	|I_2| + |I_3| &\le C(u, p) \big( h + h^{2\alpha - 1/2} + \epsilon \big) \big( C(u, p)h + \vertiii{v_h - u_h}_{V_h} \big) \\
		&\le C(u, p) \big( h^\alpha + \epsilon \big) \big( C(u, p)h + \vertiii{v_h - u_h}_{V_h} \big),
\end{align*}
where we note that $\max\{h, h^{2\alpha - 1/2}\} \le h^\alpha$ by definition of $\alpha$.

Combining the estimates above, we deduce that
\begin{equation*}
	\vertiii{v_h - u_h}_{V_h}^2 \le C(u,p)^2(h^\alpha + \epsilon)^2 + C(u,p)(h^\alpha + \epsilon)\vertiii{v_h - u_h}_{V_h},
\end{equation*}
from which \eref{eq: goal of H1 error estimate} follows.  This completes the proof of \tref{thm: H1 error estimate}.

\begin{rem}
	As for error estimation of the Lagrange multiplier, from \eref{eq: H1 error proof 3} we have
	\begin{equation} \label{eq: error estimate for Lagrange multiplier}
		\Big( \sum_{e\in\mathcal E_h^\partial} h_e \|\Pi_h^\partial \tilde\lambda + k_h - \lambda_h\|_{L^2(e)}^2 \Big)^{1/2} + \|\Pi_h^\partial \tilde\lambda + k_h - \lambda_h\|_{-1/2, \Lambda_h} \le C(u,p) (h^\alpha + \epsilon).
	\end{equation}
	This combined with \eref{eq: stability of kh} especially implies the stability
	\begin{equation} \label{eq: stability for Lagrange multiplier}
		\Big( \sum_{e\in\mathcal E_h^\partial} h_e \|\lambda_h\|_{L^2(e)}^2 \Big)^{1/2} + \|\lambda_h\|_{-1/2, \Lambda_h} \le C(u,p) (1 + \epsilon^{-1}h^2).
	\end{equation}
	From \rref{rem: stability}(i), the dependency of $\epsilon^{-1}$ may be omitted if $g = 0$.
\end{rem}

\section{$L^2$-error estimate} \label{sec6}
For the $L^2$-error analysis we need another consistency error estimate as follows:
\begin{lem} \label{lem: another estimate of Res(v)}
	In addition to the hypotheses of \lref{lem: asymptotic Galerkin orthogonality}, let $w \in H^2(\Omega)^N$ satisfy $\mathrm{div}\, w = 0$ in $\Omega$ and $w\cdot n = 0$ on $\Gamma$.
	Then we obtain
	\begin{equation*}
		|\mathrm{Res}(w)| \le Ch^{2\alpha}(\|u\|_{H^2(\Omega)} + \|p\|_{H^1(\Omega)}) \|w\|_{H^2(\Omega)}.
	\end{equation*}
\end{lem}
\begin{proof}
	We introduce a signed integration over the boundary-skin layer by
	\begin{equation*}
		(f, g)_{\Omega_h\triangle\Omega}' := (f, g)_{\Omega_h\setminus\Omega} - (f, g)_{\Omega\setminus\Omega_h}.
	\end{equation*}
	Then it follows from integration by parts that for all $v \in H^1(\tilde\Omega)^N$
	\begin{equation*}
		(\sigma(\tilde u, \tilde p)n_h, v)_{\Gamma_h} - (\sigma(u, p)n, v)_{\Gamma} = \frac\nu2 (\mathbb E(\tilde u), \mathbb E(\tilde v))'_{\Omega_h \triangle \Omega} - (\tilde p, \mathrm{div}\, v)'_{\Omega_h \triangle \Omega} + (\nu\Delta\tilde u + \nu\nabla\operatorname{div}\tilde u - \nabla\tilde p, v)'_{\Omega_h \triangle \Omega}.
	\end{equation*}
	Substituting this formula into \eref{eq: definition of Res(v)}, recalling $\sigma(u, p)n = \tau - \lambda n$ on $\Gamma$, and noting that $\jump{v} = 0$ on each $e \in \mathring{\mathcal E}_h$, we obtain
	\begin{align}
		\mathrm{Res}(v) &= (\tilde u - \tilde f, v)'_{\Omega_h \triangle \Omega} + \frac\nu2 (\mathbb E(\tilde u), \mathbb E(v))'_{\Omega_h \triangle \Omega} - (\tilde p, \mathrm{div}\, v)'_{\Omega_h \triangle \Omega} \notag \\
		&\hspace{1cm} + (\tau, v)_\Gamma - (\tilde\tau, v)_{\Gamma_h} - (\lambda, v\cdot n)_\Gamma + c_h(\tilde\lambda, v\cdot n_h). \label{eq: another expression of Res(v)}
	\end{align}
	We now take $v = \tilde w$ and apply \eref{eq: boundary-skin estimate 1}--\eref{eq: boundary-skin estimate 3} to see that all the terms but the last one on the right-hand side of \eref{eq: another expression of Res(v)} can be bounded by $Ch^2 (\|u\|_{H^2(\Omega)} + \|p\|_{H^1(\Omega)}) \|w\|_{H^2(\Omega)}$.
	The last term is then estimated by \eref{eq1: L2 error of u cdot nh - g}, which completes the proof.
\end{proof}

\begin{proof}[Proof of \tref{thm: L2 error estimate}]
	In what follows, we abbreviate the quantity $C(\|\tilde u\|_{H^2(\tilde\Omega)} + \|\tilde p\|_{H^1(\tilde\Omega)})$ just as $C(u,p)$.
	Let $\varphi \in C_0^\infty(\Omega_h)^N$ be such that $\|\varphi\|_{L^2(\Omega_h)} = 1$ and estimate $(\tilde u - u_h, \varphi)_{\Omega_h}$.
	Let $(w, r) \in H^2(\Omega)^N \times H^1(\Omega)$ be the solution of the following dual problem ($\varphi$ is extended by 0 outside $\Omega_h$):
	\begin{equation*}
		\left\{
		\begin{aligned}
			w - \nu\Delta w + \nabla r &= \varphi \quad\text{in}\quad \Omega, \\
			\mathrm{div}\,w &= 0 \quad\text{in}\quad \Omega, \\
			w\cdot n &= 0 \quad\text{on}\quad \Gamma, \\
			(\mathbb I - n\otimes n) \sigma(w, r)n &= 0 \quad\text{on}\quad \Gamma.
		\end{aligned}
		\right.
	\end{equation*}
	Then we see that $\|w\|_{H^2(\Omega)} + \|r\|_{H^1(\Omega)} \le C$.
	Setting $w_h := \Pi_h \tilde w$, we find from integration by parts and from \eref{eq: Galerkin orthogonality} that
	{\allowdisplaybreaks
	\begin{align*}
		(\tilde u - u_h, \varphi)_{\Omega_h} &= (\tilde u - u_h, \varphi)_{\Omega_h \cap \Omega} + (\tilde u - u_h, \varphi)_{\Omega_h\setminus\Omega} \\
		&= (\tilde u - u_h, \tilde w - \nu\Delta\tilde w -\nu\nabla\operatorname{div}\tilde w + \nabla\tilde r)_{\Omega_h} - (\tilde u - u_h, \tilde w - \nu\Delta\tilde w -\nu\nabla\operatorname{div}\tilde w + \nabla\tilde r - \varphi)_{\Omega_h \setminus \Omega} \\
		&= a_h(\tilde u - u_h, \tilde w) - \sum_{e\in\mathring{\mathcal E}_h} (\jump{\tilde u - u_h}, \sigma(\tilde w, \tilde r)n_e)_e - (\tilde u - u_h, \sigma(\tilde w, \tilde r)n_h)_{\Gamma_h} \\
			&\hspace{1cm} - (\tilde u - u_h, \tilde w - \nu\Delta\tilde w -\nu\nabla\operatorname{div}\tilde w + \nabla\tilde r - \varphi)_{\Omega_h \setminus \Omega} \\	
		&= a_h(\tilde u - u_h, \tilde w - w_h) + \sum_{e\in\mathring{\mathcal E}_h} (\jump{u_h}, \sigma(\tilde w, \tilde r)n_e)_e - (\tilde u - u_h, \sigma(\tilde w, \tilde r)n_h)_{\Gamma_h} \\
			&\hspace{1cm} + \mathrm{Res}(w_h) - b_h(R_h\tilde p + k_h - p_h, w_h) - c_h(\tilde\lambda + k_h - \lambda_h, w_h\cdot n_h) + j_h(u_h, w_h) \\
			&\hspace{1cm} - (\tilde u - u_h, \tilde w - \nu\Delta\tilde w -\nu\nabla\operatorname{div}\tilde w + \nabla\tilde r - \varphi)_{\Omega_h \setminus \Omega} \\
		&= a_h(\tilde u - u_h, \tilde w - w_h) + \sum_{e\in\mathring{\mathcal E}_h} (\jump{u_h}, \sigma(\tilde w, \tilde r)n_e)_e  - (\tilde u - u_h, \sigma(\tilde w, \tilde r)n_h)_{\Gamma_h} \\
			&\hspace{1cm} - (\tilde u - u_h, \tilde w - \nu\Delta\tilde w -\nu\nabla\operatorname{div}\tilde w + \nabla\tilde r - \varphi)_{\Omega_h \setminus \Omega} - j_h(u_h, \tilde w - w_h) \\
			&\hspace{1cm} + \mathrm{Res}(w_h) - b_h(R_h\tilde p + k_h - p_h, \tilde w) - (\Pi_h^\partial \tilde\lambda + k_h - \lambda_h, \Pi_h^\partial \tilde w \cdot n_h)_{\Gamma_h} \\
		&=: \sum_{i=1}^{8} I_i,
	\end{align*}
	}where we made use of the fact that $b_h(q_h, \Pi_h\tilde w) = b_h(q_h, \tilde w)$ for $q_h \in Q_h$ in the fifth equality.
	
	Let us bound each term of $I_1, \dots, I_{9}$.
	By interpolation estimates and \lref{lem: a boundary-skin estimate for nonconforming approximation}, one has $|I_1 + I_2 + I_4 + I_5| \le C(u,p) h \vertiii{\tilde u - u_h}_{V_h}$.
	It follows from \lref{lem: another estimate of Res(v)} and \lref{lem: asymptotic Galerkin orthogonality}(ii) that
	\begin{align*}
		|I_6| \le |\mathrm{Res}(\tilde w)| + |\mathrm{Res}(\tilde w - w_h)| \le C(u,p)h^{2\alpha} \|w\|_{H^2(\Omega)} + C(u,p)h \vertiii{\tilde w - w_h}_{V_h} \le C(u,p)h^{2\alpha}.
	\end{align*}
	For $I_7$, the pressure error estimate obtained in \tref{thm: H1 error estimate} gives
	\begin{equation*}
		|I_7| = |(R_h\tilde p + k_h - p_h, \operatorname{div}\tilde w)_{\Omega_h\setminus\Omega}| \le \|R_h\tilde p + k_h - p_h\|_{Q_h} Ch \|\tilde w\|_{H^2(\Omega_h)} \le C(u,p)h^2 + Ch \vertiii{\tilde u - u_h}.
	\end{equation*}
	For $I_8$, as a result of \eref{eq: error estimate for Lagrange multiplier} and \eref{eq2: L2 error of u cdot nh - g} we have
	\begin{align*}
		|I_8| &\le \Big( \sum_{e\in\mathcal E_h^\partial} h_e \|\Pi_h^\partial \tilde\lambda + k_h - \lambda_h\|_{L^2(e)}^2 \Big)^{1/2} \times Ch^{2\alpha - 1/2} \|\tilde w\|_{H^2(\tilde\Omega)} \\
			&\le C(u,p) (h + \vertiii{\tilde u - u_h}_{V_h}) h^{2\alpha - 1/2}.
	\end{align*}
	
	It remains to estimate $I_3$.
	Setting $\tilde\mu := \sigma(\tilde w, \tilde r)n_h\cdot n_h$ and $\mu_h := \Pi_h^\partial \tilde\mu$, we obtain
	\begin{align*}
		-I_3 &= ( \tilde u - u_h, (\mathbb I - n_h\otimes n_h)\sigma(\tilde w, \tilde r)n_h )_{\Gamma_h} + ( (\tilde u - u_h)\cdot n_h, \tilde\mu )_{\Gamma_h} =: I_{31} + ( (\tilde u - u_h)\cdot n_h, \tilde\mu )_{\Gamma_h} \\
		&= I_{31} + ( (\tilde u - u_h)\cdot n_h, \tilde\mu - \mu_h )_{\Gamma_h} +  ( (\Pi_h^\partial \tilde u - u_h)\cdot n_h, \mu_h )_{\Gamma_h} \\
		&= I_{31} + ( (\tilde u - u_h)\cdot n_h, \tilde\mu - \mu_h )_{\Gamma_h} + ( \Pi_h^\partial (\tilde u\cdot n_h - \tilde g), \mu_h )_{\Gamma_h} - \epsilon c_h(\lambda_h, \mu_h) \\
		&=: I_{31} + I_{32} + I_{33} + I_{34}.
	\end{align*}
	Since $(\mathbb I - n\otimes n)\sigma(w, r)n = 0$ on $\Gamma$, $\|n\circ\pi - n_h\|_{L^\infty(\Gamma_h)} \le Ch$, and $\|\sigma(\tilde w, \tilde r) - \sigma(w, r)\circ\pi\|_{L^2(\Gamma_h)} \le C\delta^{1/2} \|\nabla\sigma(\tilde w, \tilde r)\|_{L^2(\Gamma(\delta))}$, we have
	\begin{equation*}
		|I_{31}| \le Ch \|\tilde u - u_h\|_{L^2(\Gamma_h)} \le Ch \|\tilde u - u_h\|_{L^2(\Omega_h)}^{1/2} \|\tilde u - u_h\|_{V_h}^{1/2}.
	\end{equation*}
	For $I_{32}$ we get
	\begin{equation*}
		|I_{32}| \le C \|\tilde u - u_h\|_{L^2(\Gamma_h)} \|\sigma(\tilde w, \tilde r) - \Pi_h^\partial \sigma (\tilde w, \tilde r)\|_{L^2(\Gamma_h)} \le Ch^{1/2} \|\tilde u - u_h\|_{L^2(\Omega_h)}^{1/2} \|\tilde u - u_h\|_{V_h}^{1/2}.
	\end{equation*}
	By \eref{eq1: L2 error of u cdot nh - g}, $|I_{33}| \le C(u,p) h^{2\alpha} \|\mu\|_{L^2(\Gamma_h)} \le C(u,p) h^{2\alpha}$.
	From \eref{eq: stability for Lagrange multiplier} and \lref{lem: compatibility with usual H1/2 norm when nh is involved} it follows that
	\begin{equation*}
		|I_{34}| \le \epsilon \|\lambda_h\|_{-1/2, \Lambda_h} \|\mu_h\|_{1/2, \Lambda_h} \le C(u,p)(\epsilon + h^2) \|\sigma(\tilde w, \tilde r)\|_{H^{1/2}(\Gamma_h)} \le C(u,p)(\epsilon + h^2).
	\end{equation*}
	Consequently,
	\begin{equation*}
		|I_3| \le Ch^{1/2} \|\tilde u - u_h\|_{L^2(\Omega_h)}^{1/2} \|\tilde u - u_h\|_{V_h}^{1/2} + C(u,p)(h^{2\alpha} + \epsilon).
	\end{equation*}
	
	Recalling $\varphi$ is arbitrary, collecting the estimates above, and substituting the result of the $H^1$-error estimate, we deduce that
	\begin{align*}
		\|\tilde u - u_h\|_{L^2(\Omega_h)} &\le Ch\vertiii{\tilde u - u_h}_{V_h} + (C(u,p)h + C\vertiii{\tilde u - u_h}_{V_h}) h^{2\alpha - 1/2} \\
			&\hspace{1cm} + Ch^{1/2} \|\tilde u - u_h\|_{L^2(\Omega_h)}^{1/2} \|\tilde u - u_h\|_{V_h}^{1/2} + C(u,p)(h^{2\alpha} + \epsilon) \\
			&\le \frac12 \|\tilde u - u_h\|_{L^2(\Omega_h)} + C(u,p)h(h^\alpha + \epsilon) + C(u,p)h^{2\alpha - 1/2} (h^\alpha + \epsilon) + C(u,p)(h^{2\alpha} + \epsilon),
	\end{align*}
	which concludes $\|\tilde u - u_h\|_{L^2(\Omega_h)} \le C(u,p)(h^{2\alpha} + \epsilon)$.
\end{proof}

\section{Numerical results} \label{sec7}
In this section, we present numerical results using the proposed scheme \eqref{eq: discrete 2-field formulation} in two- and three- dimensional cases to validate our theoretical results.
The same test problems as in \cite{KOZ16} are considered.
In the following, we set $\nu = 1$ and use unstructured meshes.
All computations here were done with FEniCS \cite{fenics}.

\subsection{Two-dimensional case}
 We consider the problem \eqref{eq: Stokes slip BC} where  the domain $\Omega$
is the unit disk, i.e., $\Omega = \{ (x,y) \in \mathbb{R}^2 : x^2+y^2 < 1\}$.
The data $f$, $g$, and $\tau$ are chosen so that the exact solution is
 \begin{align*}
  & u(x,y) = (-y(x^2+y^2), x(x^2+y^2))^\top, \\
  & p(x,y) = 8xy.
 \end{align*}
We set the parameters as $\epsilon = 0.1h^2$ and $\gamma = 2$.
Table \ref{tb1} shows the history of convergence for the velocity and pressure.
We observe that our method achieves optimal orders in all cases, which is
in full agreement with Theorem \ref{thm: H1 error estimate} with $\alpha = 1$.

\begin{table}[ht]
	\caption{Convergence history in the two-dimensional case}
	\label{tb1}
	\centering
	\begin{tabular}{cccccccc} \hline
		& \multicolumn{2}{c}{$\|u-u_h\|_{L^2(\Omega_h)}$} & \multicolumn{2}{c}{$\|u-u_h\|_{H^1(\Omega_h)}$} & \multicolumn{2}{c}{$\|p-p_h\|_{L^2(\Omega_h)}$} \\ \hline
		$h$ & Error & Order & Error & Order & Error & Order \\ \hline
		0.1734 & 3.85E-02 & -- & 2.49E-01 & -- & 2.48E-01 & -- \\
		0.0857 & 9.59E-03 & 1.97  & 1.17E-01 & 1.07 & 1.21E-01 & 1.02 \\
		0.0459 & 2.53E-03 & 2.13  & 5.94E-02 & 1.09 & 6.21E-02 & 1.06 \\
		0.0232 & 6.46E-04 & 2.00  & 2.98E-02 & 1.01 & 3.13E-02 & 1.00 \\ \hline
	\end{tabular}
\end{table}

\subsection{Three-dimensional case}
In this example, the problem with $\Omega = \{(x,y,z)\in\mathbb{R}^3: x^2+y^2+z^2<1\}$ is considered.
The data $f$, $g$ and $\tau$ are chosen so that the exact
solution becomes
\begin{align*}
 & u(x,y,z) = (10x^2yz(y-z), 10xy^2z(z-x), 10xyz^2(x-y))^\top, \\
 & p(x,y,z) = 10xyz(z+y+z).
\end{align*}
We set $\epsilon=0.1h$ and $\gamma=5$.
The history of convergence is displayed in Table \ref{tb2}.
From the result, we see that all the orders seem to be one.
The order of the $L^2$ error of velocity coincides with Theorem \ref{thm: L2 error estimate} where $\alpha=1/2$.
 On the other hand, the $H^1$ and $L^2$ errors of velocity and pressure, respectively,
 converge with the optimal order, which is faster than expected
 in Theorem \ref{thm: L2 error estimate}.
\begin{table}[hb]
	\caption{Convergence history in the three-dimensional case}
	\label{tb2}
	\centering
	\begin{tabular}{ccccccc} \hline
		& \multicolumn{2}{c}{$\|u-u_h\|_{L^2(\Omega_h)}$} & \multicolumn{2}{c}{$\|u-u_h\|_{H^1(\Omega_h)}$} & \multicolumn{2}{c}{$\|p-p_h\|_{L^2(\Omega_h)}$} \\ \hline
		$h$ & Error & Order & Error & Order & Error & Order \\ \hline
		0.1853 & 8.62E-02 &  -- & 7.88E-01 &  -- & 5.39E-01 & -- \\
		0.0959 & 4.72E-02 & 0.87  & 4.08E-01 & 0.95  & 3.02E-01 & 0.84  \\
		0.0679 & 3.42E-02 & 0.79  & 2.86E-01 & 0.87  & 2.19E-01 & 0.79  \\
		0.0500 & 2.56E-02 & 1.01  & 2.12E-01 & 1.04  & 1.65E-01 & 1.00  \\ \hline
	\end{tabular}
\end{table}

It is noted that Krylov linear solvers, such as GMRES and BiCGSTAB methods, 
fail to solve the resulting system of linear equations when $\epsilon$ is very small.
We do not here present the numerical result because it is similar to that shown in \cite[Table 3]{KOZ16}.

\appendix
\section{Proofs of Lemmas in \sref{sec3}}
\subsection{Proof of \lref{lem: compatibility with usual H1/2 norm}} \label{sec: A.1}
In view of the definition of $\|\cdot\|_{1/2,\Lambda_h}$, the lemma is reduced to:
\begin{align}
	\|E_h^\partial \Pi_h^\partial \mu\|_{H^{1/2}(\Gamma_h)} &\le C \|\mu\|_{H^{1/2}(\Gamma_h)}, \label{eq: compatibility of discrete H1/2 Step 1} \\
	\sum_{e \in \mathcal E_h^\partial} \sum_{e' \in \mathcal E_h^\partial(e)} h_e^{N-2} |\Pi_h^\partial \mu(m_{e'}) - \Pi_h^\partial \mu(m_e)|^2& \le C \|\mu\|_{H^{1/2}(\Gamma_h)}, \label{eq: compatibility of discrete H1/2 Step 2}
\end{align}
which will be established in the following Steps 1 and 2 respectively.

\textbf{Step 1.} It is sufficient to show that the operator $E_h^\partial \Pi_h^\partial$ is stable in $L^2(\Gamma_h)$ and $H^1(\Gamma_h)$.
Then \eref{eq: compatibility of discrete H1/2 Step 1} follows by interpolation.
To this end, for $x \in e \in \mathcal E_h^\partial$ we calculate
\begin{equation*}
	E_h^\partial \Pi_h^\partial \mu(x) = \sum_{p\in\mathcal V_h(e)} \bigg( \frac1{\#\mathcal E_h^\partial(p)} \sum_{e'\in\mathcal E_h^\partial(p)} \frac1{|e'|} \int_{e'} \mu \, ds \bigg) \bar\phi_p(x) =: \sum_{p\in\mathcal V_h(e)} A_p \bar\phi_p(x).
\end{equation*}
The H\"older and Cauchy--Schwarz inequalities give
\begin{equation*}
	|A_p| \le C \sum_{e'\in\mathcal E_h^\partial(p)} |e'|^{-1/2} \|\mu\|_{L^2(e')} \le C h_e^{-(N-1)/2} \|\mu\|_{L^2(\triangle_e)},
\end{equation*}
where $\triangle_e := \bigcup\mathcal E_h^\partial(e)$ stands for a macro element of $e$.
Hence we obtain
\begin{equation*}
	\|E_h^\partial \Pi_h^\partial \mu\|_{L^2(e)}^2 = \int_e \big| \sum_{p\in\mathcal V_h(e)} A_p \bar\phi_p \big|^2 \, ds \le C \max_{p\in\mathcal V_h(e)} |A_p|^2 \max_{p\in\mathcal V_h(e)} \|\bar\phi_p\|_{L^2(e)}^2 \le C\|\mu\|_{L^2(\triangle_e)}^2,
\end{equation*}
which, after the summation for $e\in\mathcal E_h^\partial$, implies the $L^2$-stability.

For the $H^1$-stability, noting that $\sum_{p\in\mathcal V_h(e)} \bar\phi_p(x) = 1$ for $x \in e \in \mathcal E_h^\partial$, we have
\begin{equation*}
	\nabla_e E_h^\partial \Pi_h^\partial \mu(x) = \sum_{p\in\mathcal V_h(e)} \bigg( \frac1{\#\mathcal E_h^\partial(p)} \sum_{e'\in\mathcal E_h^\partial(p)} \frac1{|e'|} \int_{e'} (\mu - \theta) \, ds \bigg) \nabla_e \bar\phi_p(x) \qquad \forall \theta \in P_0(\triangle_e).
\end{equation*}
where the $\nabla_e$ means the surface gradient along $e$.
By a calculation similar to the one above, we get
\begin{equation} \label{eq3: compatibility of discrete H1/2 norm}
	\|\nabla_e E_h^\partial \Pi_h^\partial \mu\|_{L^2(e)}^2 \le Ch_e^{-2} \|\mu - \theta\|_{L^2(\triangle_e)}^2 \qquad \forall \theta \in P_0(\triangle_e).
\end{equation}
Now the Bramble--Hilbert theorem yields $\inf_{\theta\in P_0(\triangle_e)} \|\mu - \theta\|_{L^2(\triangle_e)} \le Ch_e |\mu|_{H^1(\triangle_e)}$ (see the remark below for more details).
Therefore, the $H^1$-stability is obtained, and, as we noticed earlier, this proves \eref{eq: compatibility of discrete H1/2 Step 1}.

\textbf{Step 2.} We notice that
\begin{align*}
	|\Pi_h^\partial \mu(m_{e}) - \Pi_h^\partial \mu(m_{e'})|^2
		&= \left| \frac1{|e| |e'|} \int_{e\times e'} \big(\mu(x) - \mu(y)\big) \, ds(x) ds(y) \right|^2 \\
		&\le \frac1{|e| |e'|} \int_{e\times e'} |\mu(x) - \mu(y)|^2 \, ds(x) ds(y) \\
		&\le Ch_e^{-2(N-1)} \times Ch_e^N \int_{e\times e'} \frac{|\mu(x) - \mu(y)|^2}{|x - y|^N} \, ds(x) ds(y) \\
		&\le Ch_e^{2-N} \|\mu\|_{H^{1/2}(\triangle_e)}^2,
\end{align*}
where we have used $|x - y| \le Ch_e$ for $x\in e \in\mathcal E_h^\partial$ and $y\in e'\in\mathcal E_h^\partial(e)$.
Then \eref{eq: compatibility of discrete H1/2 Step 2} follows by taking summation for $e$.

\begin{rem}
	The Bramble--Hilbert theorem used after \eref{eq3: compatibility of discrete H1/2 norm} may be justified as follows.
	Adopting the notation of local coordinates introduced in \cite[Section 8]{KOZ16}, we may assume that $\Delta_e$ is contained in some local coordinate neighborhood $U$ such that $\Gamma_h \cap U$ admits a graph representation $(y', \phi_h(y'))$.
	Let $B: \mathbb R^N \to \mathbb R^{N-1}; \; (y', y_N) \mapsto y'$ denote the projection to the base set and $\Delta_e' := B(\Delta_e)$.
	We find that the norms $\|f\|_{L^2(\Delta_e)}$ and $\|\nabla_{\Gamma_h} f\|_{L^2(\Delta_e)}$ are equivalent to $\|f'\|_{L^2(\Delta_e')}$ and $\|\nabla_{y'}f'\|_{L^2(\Delta_e')}$ respectively,
	where $f' = f'(y')$ refers to the local coordinate representation of a function $f$ given on $\Gamma_h$, and $\nabla_{\Gamma_h}$ is the surface gradient along $\Gamma_h$.
	Then the desired inequality is reduced to show
	\begin{equation*}
		\inf_{\theta' \in P_0(\Delta_e')} \|\mu' - \theta'\|_{L^2(\Delta_e')} \le Ch_e \|\nabla_{y'} \mu'\|_{L^2(\Delta_e')},
	\end{equation*}
	which indeed follows from \cite[Lemma 4.3.8]{BrSc07} together with the regularity of the meshes (note that $\operatorname{diam} \Delta_e' \le Ch_e$ and that $\Delta_e'$ is star-shaped with respect to the inscribed ball of $e'$, whose radius is greater than $\rho_{T_e}$).
\end{rem}

\subsection{Proof of \lref{lem: compatibility with usual H1/2 norm when nh is involved}} \label{sec: A.2}
Below we only prove the scalar case, since the other two cases may be treated similarly.
We first notice, from \lref{lem: compatibility with usual H1/2 norm}, that $\| \Pi_h^\partial(\mu\, n\circ\pi) \|_{1/2, \Lambda_h} \le C \|\mu\, n\circ\pi\|_{H^{1/2}(\Gamma_h)} \le C \|\mu\|_{H^{1/2}(\Gamma_h)}$.
Hence it remains to deal with $\|\Pi_h^\partial (\mu(n_h - n\circ\pi))\|_{1/2, \Lambda_h}$, and, in view of the definition of $\|\cdot\|_{1/2, \Lambda_h}$, it suffices to show the following:
\begin{align}
	\|E_h^\partial \Pi_h^\partial (\mu(n_h - n\circ\pi))\|_{H^{1/2}(\Gamma_h)} &\le Ch^{1/2} \|\mu\|_{L^2(\Gamma_h)}, \label{eq: stability estimate 1 when nh appears} \\
	\sum_{e \in \mathcal E_h^\partial} \sum_{e' \in \mathcal E_h^\partial(e)} h_e^{N-2} \bigg| \frac1{|e|} \int_e \mu(n_h - n\circ\pi)\, ds - \frac1{|e'|} \int_{e'} \mu(n_h - n\circ\pi)\, ds \bigg|^2 &\le Ch \|\mu\|_{L^2(\Gamma_h)}^2. \label{eq: stability estimate 2 when nh appears}
\end{align}

Estimate \eref{eq: stability estimate 1 when nh appears} follows by interpolation if we establish
\begin{align*}
	\|E_h^\partial \Pi_h^\partial (\mu(n_h - n\circ\pi))\|_{L^2(\Gamma_h)} &\le Ch \|\mu\|_{L^2(\Gamma_h)}, \\
	\|E_h^\partial \Pi_h^\partial (\mu(n_h - n\circ\pi))\|_{H^1(\Gamma_h)} &\le C \|\mu\|_{L^2(\Gamma_h)}.
\end{align*}
In fact, for $x\in e\in\mathcal E_h^\partial$ we have
\begin{align*}
	E_h^\partial \Pi_h^\partial [\mu (n_h - n\circ\pi)](x) = \sum_{p \in \mathcal V_h(e)} \frac1{\#\mathcal E_h^\partial(p)} \Big( \sum_{e' \in \mathcal E_h^\partial(p)} \frac1{|e|} \int_e \mu(n_h - n\circ\pi)\, ds \Big) \bar\phi_p(x).
\end{align*}
Noting that $\|n_h - n\circ\pi\|_{L^\infty(e)} \le Ch_e$, we obtain
\begin{align*}
	\|E_h^\partial \Pi_h^\partial (\mu (n_h - n\circ\pi))\|_{L^2(e)}^2 &\le C \sum_{e'\in\mathcal E_h^\partial(e)} |e|^{-1} \|\mu\|_{L^2(e)}^2 h_e^2 \times \sup_{p\in\mathcal V_h(e)} \|\bar\phi_p\|_{L^2(e)}^2 \\
		&\le C \sum_{e'\in\mathcal E_h^\partial(e)} h_e^{3-N} \|\mu\|_{L^2(e)}^2 \times h_e^{N-1} \le C h^2 \|\mu\|_{L^2(\triangle_e)}^2,
\end{align*}
which, after the summation for $e \in \mathcal E_h^\partial$, implies the $L^2$-estimate.
One can obtain the $H^1$-estimate in a similar way, and thus \eref{eq: stability estimate 1 when nh appears} is proved.

Finally, a direct computation shows that the left-hand side of \eref{eq: stability estimate 2 when nh appears} is bounded by
\begin{equation*}
	C\sum_{e \in \mathcal E_h^\partial} \sum_{e'\in\mathcal E_h^\partial(e)} h_e^{N-2} \big( |e|^{-1} \|\mu\|_{L^2(e)}^2 h_e^2 + |e'|^{-1} \|\mu\|_{L^2(e')}^2 h_{e'}^2 \big) \le Ch \|\mu\|_{L^2(\Gamma_h)}^2.
\end{equation*}
This completes the proof of \lref{lem: compatibility with usual H1/2 norm when nh is involved}.

\subsection{Proof of \lref{lem: discrete lifting theorem}} \label{sec: A.3}
By the standard lifting theorem, there exists a linear operator $L_h : H^{1/2}(\Gamma_h)^N \to H^1(\Omega_h)^N$ such that $(L_h\psi)|_{\Gamma_h} = \psi$ and $\|L_h\psi\|_{H^1(\Omega_h)} \le C\|\psi\|_{H^{1/2}(\Gamma_h)}$.
We then define $v_h \in V_h$ by
\begin{equation*}
	v_h(m_e) =
	\begin{cases}
		\left[ \Pi_h L_h E_h^\partial (\mu_h n_h) \right](m_e) &\text{for } e\in\mathring{\mathcal E}_h, \\
		(\mu_h n_h)(m_e) &\text{for } e\in\mathcal E_h^\partial.
	\end{cases}
\end{equation*}
It is clear that $v_h \cdot n_h = \mu_h$ at all $m_e$'s lying on $\Gamma_h$.
We prove \eref{eq: discrete lift stability} in the following three steps.

\textbf{Step 1.} Let us show
\begin{equation} \label{eq: error between vh and naive lift}
	\|v_h - \Pi_h L_h E_h^\partial (\mu_h n_h)\|_{V_h} \le C\|\mu_h\|_{1/2,\Lambda_h}.
\end{equation}
Observe that
\begin{equation*}
	v_h - \Pi_h L_h E_h^\partial (\mu_h n_h) = \sum_{e\in\mathcal E_h^\partial} \left[ \mu_hn_h - \Pi_h L_h E_h^\partial (\mu_h n_h) \right](m_e) \phi_e.
\end{equation*}
By the definitions of $\Pi_h$, $L_h$, and $E_h^\partial$, for $e \in \mathcal E_h^\partial$ we obtain
\begin{align*}
	\left[ \Pi_h L_h E_h^\partial (\mu_h n_h) \right](m_e) &= \frac1{|e|} \int_e L_hE_h^\partial(\mu_hn_h) \, ds = \frac1{|e|} \int_e E_h^\partial(\mu_hn_h) \, ds \\
		&= \frac1{|e|} \int_e \sum_{p\in\mathcal V_h(e)} \frac1{\#\mathcal E_h^\partial(p)} \sum_{e'\in\mathcal E_h^\partial(p)} (\mu_hn_h)(m_{e'}) \bar\phi_p \, ds.
\end{align*}
Therefore, noting that $\sum_{p\in\mathcal V_h(e)} \bar\phi_p \equiv 1$, we deduce
\begin{align*}
	v_h - \Pi_h L_h E_h^\partial (\mu_h n_h) &= \sum_{e\in\mathcal E_h^\partial} \left[ \frac1{|e|} \sum_{p\in\mathcal V_h(e)} \frac1{\#\mathcal E_h^\partial(p)} \sum_{e'\in\mathcal E_h^\partial(p)} \Big( (\mu_hn_h)(m_{e}) - (\mu_hn_h)(m_{e'}) \Big) \int_e \bar\phi_p\, ds \right] \phi_e \\
		&=: \sum_{e\in\mathcal E_h^\partial} A_e \phi_e,
\end{align*}
where the coefficient $A_e$ can be estimated, using $|n_h(m_e) - n_h(m_{e'})| \le Ch_e$, by
\begin{equation*}
	|A_e| \le C \sum_{e'\in\mathcal E_h^\partial(e)} |\mu_h(m_e) - \mu_h(m_{e'})| + Ch_e \sum_{e'\in\mathcal E_h^\partial(e)} |\mu_h(m_{e'})|.
\end{equation*}
Then we conclude that
\begin{align*}
	\| v_h - \Pi_h L_h E_h^\partial (\mu_h n_h) \|_{V_h}^2 &= \sum_{T \in \mathcal T_h} \Big\| \sum_{e\in\mathcal E_h^\partial} A_e\phi_e \Big\|_{H^1(T)}^2 = \sum_{e\in\mathcal E_h^\partial} \|A_e \phi_e\|_{H^1(T_e)}^2 \le C \sum_{e\in\mathcal E_h^\partial} |A_e|^2 h_e^{N-2} \\
		&\le C \sum_{e\in\mathcal E_h^\partial} \sum_{e'\in\mathcal E_h^\partial(e)} h_e^{N-2} |\mu_h(m_e) - \mu_h(m_{e'})|^2 + C \sum_{e\in\mathcal E_h^\partial} \sum_{e'\in\mathcal E_h^\partial(e)} h_e^{N} |\mu_h(m_{e'})|^2.
\end{align*}
The last term on the right-hand side can be bounded by $h\|\mu_h\|_{L^2(\Gamma_h)}^2$ and this proves \eref{eq: error between vh and naive lift}.

\textbf{Step 2.} The stability properties of $\Pi_h$ and $L_h$ imply
\begin{equation*}
	\|\Pi_h L_h E_h^\partial (\mu_h n_h)\|_{V_h} \le C \|L_h E_h^\partial (\mu_h n_h)\|_{H^1(\Omega_h)} \le C\|E_h^\partial (\mu_h n_h)\|_{H^{1/2}(\Gamma_h)}.
\end{equation*}
Furthermore, by $n\circ\pi \in W^{1,\infty}(\Gamma_h)$ and by the definition of $\|\cdot\|_{1/2, \Lambda_h}$, one has
\begin{equation*}
	\|(E_h^\partial \mu_h) n\circ\pi\|_{H^{1/2}(\Gamma_h)} \le C \|E_h^\partial\mu_h\|_{H^{1/2}(\Gamma_h)} \le C\|\mu_h\|_{1/2,\Lambda_h}.
\end{equation*}
Therefore, to establish \eref{eq: discrete lift stability} it remains to prove
\begin{equation*}
	\|E_h^\partial (\mu_h n_h) - (E_h^\partial \mu_h) n\circ\pi\|_{H^{1/2}(\Gamma_h)} \le C h^{1/2} \|\mu_h\|_{L^2(\Gamma_h)}.
\end{equation*}
This estimate follows from interpolation between $L^2(\Gamma_h)$ and $H^1(\Gamma_h)$ if we prove 
\begin{align}
	\|E_h^\partial (\mu_h n_h) - (E_h^\partial \mu_h) n\circ\pi\|_{L^2(\Gamma_h)} &\le C h \|\mu_h\|_{L^2(\Gamma_h)}, \label{eq: L2 estimate for Eh muh nh} \\
	\|E_h^\partial (\mu_h n_h) - (E_h^\partial \mu_h) n\circ\pi\|_{H^1(\Gamma_h)} &\le C \|\mu_h\|_{L^2(\Gamma_h)}. \label{eq: H1 estimate for Eh muh nh}
\end{align}

\textbf{Step 3.} Let us prove \eref{eq: L2 estimate for Eh muh nh} and \eref{eq: H1 estimate for Eh muh nh}.
By the definition of $E_h^\partial$, for $x \in e \in \mathcal E_h^\partial$ we calculate
\begin{equation*}
	[E_h^\partial (\mu_h n_h) - (E_h^\partial \mu_h) n\circ\pi](x) = \sum_{p\in\mathcal V_h(e)} \frac1{\#\mathcal E_h^\partial(p)} \sum_{e'\in\mathcal E_h^\partial(p)} \mu_h(m_{e'}) \big( n_h(m_{e'}) - n\circ\pi(x) \big) \bar\phi_p(x).
\end{equation*}
Therefore,
\begin{align*}
	\| E_h^\partial (\mu_h n_h) - (E_h^\partial \mu_h) n\circ\pi \|_{L^2(e)}^2 &\le C \sum_{e'\in\mathcal E_h^\partial(e)} |\mu_h(m_{e'})|^2 \sup_{x\in e}|n_h(m_{e'}) - n\circ\pi(x)|^2  \sup_{p\in\mathcal V_h(e)} \|\bar\phi_p\|_{L^2(e)}^2 \\
		&\le C \sum_{e'\in\mathcal E_h^\partial(e)} h_e^{N+1} |\mu_h(m_{e'})|^2 \le Ch^2 \sum_{e'\in\mathcal E_h^\partial(e)} \|\mu_h\|_{L^2(e')}^2,
\end{align*}
where we have used the fact $|n_h(m_{e'}) - n\circ\pi(x)| \le Ch_e$.
Adding the above estimates for $e \in \mathcal E_h^\partial$ yields \eref{eq: L2 estimate for Eh muh nh}.
Estimate \eref{eq: H1 estimate for Eh muh nh} can be proved similarly, and this completes the proof of \lref{lem: discrete lifting theorem}.

\subsection{Proof of \lref{lem: a boundary-skin estimate for nonconforming approximation}} \label{sec: A.4}
It suffices to prove $\|v + v_h\|_{L^2(\Omega_h\setminus\Omega)} \le Ch \vertiii{v + v_h}_{V_h}$ for all $v \in H^1(\Omega_h)^N$ and $v_h \in V_h$.
We define an enriching operator $E_h: V_h \to \overline V_h$ by
\begin{equation*}
	E_hv_h = \sum_{p \in \mathcal V_h} \Big( \frac1{\#\mathcal T_h(p)} \sum_{T \in \mathcal T_h(p)} v_h|_T(p) \Big) \bar\phi_p,
\end{equation*}
where $\mathcal T_h(p) := \{ T\in\mathcal T_h \,:\, p \in T \}$ means the elements that share the vertex $p$.
In view of \eref{eq: corollary of boundary-skin estimate 3} we have
\begin{align}
	\|v + v_h\|_{L^2(\Omega_h\setminus\Omega)} &\le \|v + E_hv_h\|_{L^2(\Omega_h\setminus\Omega)} + \|v_h - E_hv_h\|_{L^2(\Omega_h\setminus\Omega)} \notag \\
		&\le Ch \|v + E_hv_h\|_{H^1(\Omega_h)} + \|v_h - E_hv_h\|_{L^2(\Omega_h\setminus\Omega)} \notag \\
		&\le Ch \|v + v_h\|_{V_h} + Ch \|v_h - E_hv_h\|_{V_h} + \|v_h - E_hv_h\|_{L^2(\Omega_h\setminus\Omega)}. \label{eq: estimate 1 for v - Ev}
\end{align}
Below we estimate the second and third terms in the right-hand side.

Since $v_h$ and $E_hv_h$ are linear for $x \in T \in \mathcal T_h$ we obtain the expression
\begin{equation*}
	v_h(x) - E_hv_h(x) = \sum_{p \in \mathcal V_h(T)} \Big( \frac1{\#\mathcal T_h(p)} \sum_{T' \in \mathcal T_h(p)} (v_h|_T - v_h|_{T'})(p) \Big) \bar\phi_p(x),
\end{equation*}
where $\mathcal V_h(T) := \mathcal V_h \cap T$ means the vertices of $T$.
Here, discontinuity at $p$ can be estimated by that across edges near $p$, that is, $\big| (v_h|_T - v_h|_{T'})(p) \big| \le \sum_{e \in \mathring{\mathcal E}_h(p)} \|\jump{v_h}\|_{L^\infty(e)}$ where $\mathring{\mathcal E}_h(p) = \{e \in \mathring{\mathcal E}_h \,:\, p \in e\}$ stands for the interior edges sharing the vertex $p$ (cf.\ \cite[p.\ 1073]{Bre03}).
Therefore,
\begin{align*}
	\|\nabla(v_h - E_hv_h)\|_{L^2(T)}^2 &\le C \sum_{e \in \mathring{\mathcal E}_h(T)} \|\jump{v_h}\|_{L^\infty(e)}^2 \sup_{p\in\mathcal V_h(T)} \|\nabla\bar\phi_p\|_{L^2(T)}^2 \\
		&\le C \sum_{e \in \mathring{\mathcal E}_h(T)} h_e^{-N+1} \|\jump{v_h}\|_{L^2(e)}^2 \times Ch_T^{N-2} \\
		&\le C \sum_{e \in \mathring{\mathcal E}_h(T)} h_e^{-1} \|\jump{v_h}\|_{L^2(e)}^2,
\end{align*}
where $\mathring{\mathcal E}_h(T) = \{ e \in \mathring{\mathcal E}_h \,:\, e \subset T \}$ means the faces of $T$ that are inside $\Omega_h$.
$\|v_h - E_hv_h\|_{L^2(T)}$ can be estimated in a similar manner, and adding these estimates for $T \in \mathcal T_h$ yields
\begin{equation} \label{eq: estimate 2 for v - Ev}
	\|v_h - E_hv_h\|_{V_h} \le C \Big( \sum_{e \in \mathring{\mathcal E}_h} h_e^{-1} \|\jump{v_h}\|_{L^2(e)}^2 \Big)^{1/2}.
\end{equation}

For the third term one has
\begin{align*}
	\|v_h - E_hv_h\|_{L^2(\Omega_h\setminus\Omega)}^2 \le \sum_{e \in \mathcal E_h^\partial} |T_e\setminus\Omega| \|v_h - E_hv_h\|_{L^\infty(T_e)}^2,
\end{align*}
where $|T_e\setminus\Omega|$ denotes the $N$-dimensional measure of $T_e\setminus\Omega$ and is bounded by $C h_e^{N-1}\delta_e$.  It follows that
\begin{equation*}
	\|v_h - E_hv_h\|_{L^\infty(T_e)}^2 \le C \sum_{e' \in \mathring{\mathcal E}_h(T_e)} \|\jump{v_h}\|_{L^\infty(e')}^2 \sup_{p\in\mathcal V_h(T_{e'})} \|\bar\phi_p\|_{L^\infty(T_e)}^2 \le C \sum_{e' \in \mathring{\mathcal E}_h(T_e)} h_{e'}^{-N+1} \|\jump{v_h}\|_{L^2(e')}^2.
\end{equation*}
We thus obtain
\begin{equation} \label{eq: estimate 3 for v - Ev}
	\|v_h - E_hv_h\|_{L^2(\Omega_h\setminus\Omega)} \le C \Big( h\delta \sum_{e \in \mathring{\mathcal E}_h} h_e^{-1} \|\jump{v_h}\|_{L^2(e)}^2 \Big)^{1/2} \le Ch^{3/2} \Big( \sum_{e \in \mathring{\mathcal E}_h} h_e^{-1} \|\jump{v_h}\|_{L^2(e)}^2 \Big)^{1/2}.
\end{equation}
Combining \eref{eq: estimate 1 for v - Ev}--\eref{eq: estimate 3 for v - Ev} and noting that $[v] = 0$ on each $e \in \mathring{\mathcal E}_h$, we conclude the desired estimate.

\begin{rem}
	\lref{lem: a boundary-skin estimate for nonconforming approximation} holds for general discontinuous P1 functions as well, because we did not use the continuity at midpoints in the proof.
\end{rem}

\section*{Acknowledgements}
We would like to thank Professor Masahisa Tabata for his valuable comments which prompted us to initiate the present study.


\begin{thebibliography}{10}

\bibitem{BaDe99}
{\sc E.~B\"ansch and K.~Deckelnick}, {\em Optimal error estimates for the
  {S}tokes and {N}avier-{S}tokes equations with slip-boundary condition},
  Math.\ Mod.\ Numer.\ Anal., 33 (1999), pp.~923--938.

\bibitem{BdV04}
{\sc H.~{B}eir\~ao~da Veiga}, {\em Regularity for {S}tokes and generalized
  {S}tokes systems under nonhomogeneous slip-type boundary conditions}, Adv.
  Differential Equations, 9 (2004), pp.~1079--1114.

\bibitem{BHP05}
{\sc C.~Bernardi, F.~Hecht, and O.~Pironneau}, {\em Coupling {D}arcy and
  {S}tokes equations for porus media with cracks}, ESAIM: Math.\ Mod.\ Numer.\
  Anal., 39 (2005), pp.~7--35.

\bibitem{BBF13}
{\sc D.~Boffi, F.~Brezzi, and M.~Fortin}, {\em Mixed Finite Element Methods and
  Applications}, Springer, 2013.

\bibitem{Bre03}
{\sc S.~C. Brenner}, {\em Korn's inequalities for piecewise ${H}^1$ vecor
  fields}, Math.\ Comp., 73 (2003), pp.~1067--1087.

\bibitem{BrSc07}
{\sc S.~C. Brenner and L.~R. Scott}, {\em The mathematical theory of finite
  element methods}, Springer, 3rd~ed., 2007.

\bibitem{BSZ14}
{\sc S.~C. Brenner, L.-Y. Sung, and Y.~Zhang}, {\em A quadratic ${C}^0$
  interior penalty method for an elliptic optimal control problem with state
  constraints}, in Recent Developments in Discontinuous Galerkin Finite Element
  Methods for Partial Differential Equations, X.~F. et~al., ed., 2014,
  pp.~97--132.

\bibitem{BuHa05}
{\sc E.~Burman and P.~Hansbo}, {\em Stabilized {C}rouzeix--{R}aviart element
  for the {D}arcy--{S}tokes problem}, Numer.\ Meth.\ PDE, 21 (2005),
  pp.~986--997.

\bibitem{CaLi09}
{\sc A.~\c{C}a\u{g}lar and A.~Liakos}, {\em Weak imposition of boundary
  conditions for the {N}avier-{S}tokes equations by a penalty method}, Int.\ J.
  Numer.\ Meth.\ Fluids, 61 (2009), pp.~411--431.

\bibitem{Cia78}
{\sc P.~G. Ciarlet}, {\em The Finite Element Method for Elliptic Problems},
  North-Holland, 1978.

\bibitem{CrRa73}
{\sc M.~Crouzeix and P.-A. Raviart}, {\em Conforming and nonconforming finite
  element methods for solving the stationary {S}tokes equations {I}},
  R.A.I.R.O. Numer.\ Anal., 7 (1973), pp.~33--76.

\bibitem{DiUr15}
{\sc I.~Dione and J.~Urquiza}, {\em Penalty: finite element approximation of
  {S}tokes equations with slip boundary conditions}, Numer.\ Math., 129 (2015),
  pp.~587--610.

\bibitem{GiTr98}
{\sc D.~Gilbarg and N.~S. Trudinger}, {\em Elliptic Partial Differential
  Equations of Second Order}, Springer, 1998.

\bibitem{KaKe18}
{\sc T.~Kashiwabara and T.~Kemmochi}, {\em ${L}^\infty$- and
  ${W}^{1,\infty}$-error estimates of linear finite element method for
  {N}eumann boundary value problem in a smooth domain}, arXiv:1804.00390.

\bibitem{KOZ16}
{\sc T.~Kashiwabara, I.~Oikawa, and G.~Zhou}, {\em Penalty method with
  {P}1/{P}1 finite element approximation for the {S}tokes equations under the
  slip boundary condition}, Numer.\ Math., 134 (2016), pp.~705--740.

\bibitem{Kno99}
{\sc P.~Knobloch}, {\em Variational crimes in a finite element discretization
  of 3{D} {S}tokes equations with nonstandard boundary conditions}, East-West
  J. Numer.\ Math., 7 (1999), pp.~133--158.

\bibitem{LSY03}
{\sc W.~Layton, F.~Schieweck, and I.~Yotov}, {\em Coupling fluid flow with
  porous media flow}, SIAM J. Numer.\ Anal., 40 (2003), pp.~2195--2218.

\bibitem{fenics}
{\sc A.~Logg, K.-A. Mardal, G.~N. Wells, et~al.}, {\em Automated Solution of
  Differential Equations by the Finite Element Method}, Springer, 2012.

\bibitem{Ver87}
{\sc R.~Verf\"urth}, {\em Finite element approximation of incompressible
  {N}avier-{S}tokes equations with slip boundary condition}, Numer.\ Math., 50
  (1987), pp.~697--721.

\bibitem{ZKO16}
{\sc G.~Zhou, T.~Kashiwabara, and I.~Oikawa}, {\em Penalty method for the
  stationary {N}avier--{S}tokes problems under the slip boundary condition}, J.
  Sci.\ Comput., 68 (2016), pp.~339--374.

\bibitem{ZKO17}
\leavevmode\vrule height 2pt depth -1.6pt width 23pt, {\em A penalty method for
  the time-dependent {S}tokes problem with the slip boundary condition and its
  finite element approximation}, Appl.\ Math., 62 (2017), pp.~377--403.

\end{thebibliography}

\end{document}